\documentclass{amsart}

\usepackage{amssymb}
\usepackage{amsmath}
\usepackage{amsfonts}
\usepackage{mathtools}
\usepackage{mathrsfs}
\usepackage{fix-cm}
\usepackage{graphicx}
\usepackage{amscd}
\usepackage{enumerate}
\usepackage{mathptmx}
\usepackage{multirow,booktabs}
\usepackage{hyperref}
\usepackage{tikz}
\usetikzlibrary{arrows,decorations,shapes,shadows,cd}
\usetikzlibrary{babel}
\usetikzlibrary{matrix}

\usepackage{soul}
\sodef\spred{}{.2em}{.9em plus.4em}{1em plus.1em minus.1em}

\usepackage[latin1]{inputenc}
\usepackage[T1]{fontenc}

\usepackage[english]{babel}

\usepackage{bbm}

\usepackage[all]{xy}

\newbox\mybox
\def\overtag#1#2#3{\setbox\mybox\hbox{$#1$}\hbox to
  0pt{\vbox to 0pt{\vglue-#3\vglue-\ht\mybox\hbox to \wd\mybox
      {\hss$\ss#2$\hss}\vss}\hss}\box\mybox}
\def\undertag#1#2#3{\setbox\mybox\hbox{$#1$}\hbox to 0pt{\vbox to
    0pt{\vglue#3\vglue\ht\mybox\hbox to \wd\mybox
      {\hss$\ss#2$\hss}\vss}\hss}\box\mybox}
\def\lefttag#1#2#3{\hbox to 0pt{\vbox to 0pt{\vss\hbox to
      0pt{\hss$\ss#2$\hskip#3}\vss}}#1}
\def\righttag#1#2#3{\hbox to 0pt{\vbox to 0pt{\vss\hbox to
      0pt{\hskip#3$\ss#2$\hss}\vss}}#1}
\let\ss\scriptstyle

\def\Dot{\lower.2pt \hbox to 3.5pt{\hss$\bullet$\hss}}
\def\Circ{\lower.2pt \hbox to 3.5pt{\hss$\circ$\hss}}

\def\splicediag#1#2{\xymatrix@R=#1pt@C=#2pt@M=0pt@W=0pt@H=0pt}

\newcommand\lineto{\ar@{-}}
\newcommand\dashto{\ar@{--}}
\newcommand\dotto{\ar@{.}}

\newtheorem{theorem}{Theorem}[section]
\newtheorem{lemma}[theorem]{Lemma}
\newtheorem{proposition}[theorem]{Proposition}
\newtheorem{corollary}[theorem]{Corollary}
\newtheorem{algorithm}[theorem]{Algorithm}

\theoremstyle{definition} 
\newtheorem{definition}[theorem]{Definition} 
\newtheorem{example}[theorem]{Example}

\newcommand{\N}{\mathbbm{N}}
\newcommand{\C}{\mathbbm{C}}

\newcommand{\Z}{\mathbbm{Z}}

\newcommand{\morf}[4][\to]{ #2 \colon #3 #1 #4}

\newcommand{\im}{\operatorname{Im}}

\newcommand{\id}{\operatorname{id}}

\newcommand{\mul}{\operatorname{mult}}

\renewcommand{\phi}{\varphi}
\renewcommand{\epsilon}{\varepsilon}

\begin{document}

\title{Images of maps from $(\C^2,0)$ to $(\C^n,0)$}
\author{Helge M\o{}ller Pedersen}
\address{Departamento de Matem\'atica\\ Universidade Federal do Cear\'a\\
Campus do Pici, Bloco 914 \\ CEP 60455-760,  Fortaleza, CE\\ Brazil
}
\email{helge@mat.ufc.br}




\date{}

\begin{abstract} 
Let $\morf{F}{(\C^2,0)}{(\C^n,0)}$ be the germ of a finite map and
$(X,0)$ be its image. We will in this article using the topology of
the link show that $(X,0)$ has to be a quotient singularity if it is
normal. We will also discuss the case when the image is not normal.
Finally we describe the possible topological types, including a
discussion of the groups and examples of how to construct a map to a
given topology. 
\end{abstract}

\maketitle

\bibliographystyle{amsalpha}

\section{Introduction}

The study of germs of holomorphic maps $\morf{f}{(\C^m,0)}{(\C^n,0)}$
is a very active field of study. The case when $m=2$ is of particular
interest. Especially the case of finite determined germs have been
studied a lot. On of the great achievements of singularity theorem is the
classification of finitely determined map germs
$\morf{f}{(\C^2,0)}{(\C^3,0)}$ by Mond \cite{mond1}. Their images have
also seen great interest. It has, for example, been discovered that the
relationship between the image Milnor number and the
$\mathscr{A}$-codimension looks a lot like the relationship between
Milnor and Tjurina numbers for hypersurfaces
\cite{dejongvanstraten,dejongvanstraten,mond2}. For a
pretty recent overview of the theory see
\cite{mondnunoballesterossingularitiesofmappings}. 

Now finitely determined maps are finite and generically one to one, so
in light of this Guillermo Pe\~nafort Sanchis asked the following
question: ``Which germs of surface singularities in $\C^3$ can be the
image of a germ of a finite maps from $\C^2$?'' during the conference:  
60LB Geometry and Singularities -- 60th anniversary of Lev
Birbrair. It is clear that if one wants to try to push Mond's
classification to a more general groups of map germs that are finite and
generically $k$ to one, then one would need an answer to this question.

An obvious example of the image of a finite holomorphic map is the
$A_n$ singularity given 
by $xy=z^n$, which is clearly the image of $(s^n,t^n,st)$. Notice that
$A_n$ is a cyclic quotient singularity, that is the quotient of
$(\C^2,0)$ by $\Z/n\Z$ acting by rotation. In general if $(X,0)$ is
the quotient of $(\C^2,0)$ by an algebraic
group then it follows that the quotient map will have $(X,0)$ as its
image, though it will not necessarily be a singularity in $\C^3$ but
may instead be in $\C^n$ for some $n$.

Javier Fern\'andez de Bobadilla gave us the hint that one should
look at the restriction of the map to a preimage of the link of
$(X,0)$. Which is what we will do in this article.

The study of 3-manifolds, particularly those that arise as links
of complex surface singularities, is well-established.
By using the knowledge we have of 3-manifolds, we prove the following.

\begin{theorem}\label{mainthm}
Let $(X,0)\subset (\C^n,0)$ be the germ of a normal surface
singularity. Then the following are equivalent:
\begin{enumerate}[(I)]
\item  $(X,0)$ is  the image of a finite holomorphic map $F:(\C^2,0)\to
 (\C^n,0)$.\label{equivimageofmap} 
\item  $(X,0)$ is a  quotient of $(\C^2,0)$ by a finite subgroup of $U(2)$
 acting linearly.\label{equivquitoinetsing}
\item $\pi_1(L)$ is finite where $L$ is the link of
  $(X,0)$.\label{equivtrivialfundgroup}
\end{enumerate}
\end{theorem}
The equivalence between \eqref{equivquitoinetsing} and
\eqref{equivtrivialfundgroup} is a classical result, one direction
follows by a Theorem of Prill
\cite{PrillLocalClasificationsOfQuotients}, the other direction is
even simpler to
observe. It follows from this that the group acting in
\eqref{equivquitoinetsing} is $\pi_1(L)$, and we will list all the
possible groups and the corresponding topologies of $(X,0)$ and its
link. They all turn out to be taut and rational
singularities. Specifically in the case the image is a hypersurface,
it is one of the simple $A_k$, $D_k$, $E_6$, $E_7$ and $E_8$
singularities. We will construct maps that realize the simple
singularities as their images, and give a way to construct maps for
other groups with some examples.

The assumption of \ref{mainthm} that $(X,0)$ is normal is a very
common assumption when spaces are considered, since the normalization does
not change the topological type of an isolated singularity.
For maps it turns out that one cannot in general make the assumption
that the image is normal, any finitely determined map that is not an
embedding do in fact have image the is not normal.
Listing all the non normal singularities, that are the images
of finite holomorphic maps, seem out of reach for the moment. But we
provide the following algorithm for determining when a given surface
singularity is the image of a finite holomorphic map:

\begin{algorithm}\label{nonnomralalgo}
Let $(X,0)$ be the germ of an irreducible complex surface singularity. Then the
following steps will determine if it is the image of a finite map:
\begin{enumerate}
\item Determine the normalization $n:(\widetilde{X},0)\to (X,0)$.
\item Calculate the fundamental group of the link $\widetilde{L}$ of the
  $(\widetilde{X},0)$.\label{step2} 
\end{enumerate}
Then $(X,0)$ is the image of a finite map if and only if
$\pi_1(\widetilde{L})$ is finite.

In the case the answer is positive then one obtains a map
$F:(\C^2,0)\to (X,0)$ as $n\circ q$ where $q:(\C^2,0)\to
(\widetilde{X},0)$ is a map realizing $(\widetilde{X},0)$ as the
quotient of $(\C^2,0)$ under the action of $\pi_1(\widetilde{L})$.
\end{algorithm}In step \eqref{step2} one can instead of calculating
$\pi_1(\widetilde{L})$ determine if $(\widetilde{X},0)$ is one of the
singularities we describe in section \ref{topologyandaction}, for example by
determining the topology of $\widetilde{L}$ by obtaining its plumbing
graph as the dual graph of a good resolution. Then $q$ is given
by the maps we discus in section \ref{topologyandaction} up to a change of
coordinates.

We start by reviewing the 3-manifold topology needed to prove
Theorem \ref{mainthm} in section \ref{3manifolds}, and the proof is
then done in Section \ref{quotientsing}. In Section \ref{nonnormal} we
will consider the case of non normal image and give the proof of
\ref{nonnomralalgo}.
In Section \ref{topologyandaction} we will give a
detailed description of the topologies that can appear as the images
of finite map germs including which
groups have them as quotients. We will also construct map realizing all
the simple singularities as images of finite map germs and, show how
to construct the maps in the general case with some examples.

\subsection*{Acknowledgment}
I would like to thank Guillhermo Pe\~nafort for asking the question in
the first place and helping me with calculating the actual maps given
in section \ref{mapstobinarypolyhedral}, Ra\'ul Oset Sinha and  Javier de
Bobadilla discussing it with us, Nu\~no-Ballesteros for confirming
that one should look at the link, and Lev Birbrair, Edson Sampaio and
Andrew de Plessis for further discussion of the problem. I would also
like to thank an anonymous reviewer that pointed ot a gap in the
original proof of Lemma \ref{linktolink}, which lead to the inclusion
of Lemma \ref{bounderyofsmootcones} to close this gap.

Last but not
least I want to express how grateful I am to my former Ph.D.\ advisor
Walter Neumann who taught me all the topology of 3-manifolds used
in this article. He sadly passed away during the time it took from the
finishing of this manuscript to its publication. He will surely be
missed by many, but his memory will live on  it the wonderful
mathematics he gave us.

\section{Review of the topology of 3-manifolds}\label{3manifolds}

In this section we are going to give a short review with reference of
the three dimensional topology that is used in the proof of the main
theorem. We assume that all manifolds are orientable, compact and
without boundaries and will
therefore only state the results in these cases. This
assumption will be satisfied by the manifolds we are concerned with
in later sections, since they are boundaries of complex manifolds.

\begin{definition}
A three dimensional manifold $M$ is called \emph{irreducible} if every
embedded two sphere in $M$ bounds a ball.
\end{definition}

It is Theorem 1 of Neumann in \cite{plumbing} that the link of
a normal complex surface singularity is an irreducible 3 manifold. 

Next we have the Sphere Theorem by Papakyriakopoulos
\begin{theorem}
If $M$ is a 3-manifold and $N$ a $\pi_1(M)$-invariant subgroup of
$\pi_2(M)$ with $\pi_2(M)\setminus N\neq \emptyset$, then there is an
embedding $\morf{g}{S^2}{M}$ such that the class of $g$ in $\pi_2(M)$
is not in $N$. 
\end{theorem}(For more on this Theorem see for example Chapter 4 of
\cite{hempel}) Noticed that the trivial subgroup of $\pi_2(M)$ is
$\pi_1(M)$-invariant. So if $\pi_2(M)$ is not trivial, then the
theorem stats that there is an embedded $S^2$ representing a non
trivial element of $\pi_2(M)$. But if $M$ is irreducible then any
embedded $S^2$ bounds a ball, and hence represents the trivial element
of $\Pi_2(M)$. So we can conclude.

\begin{corollary}\label{spheretheorem} 
If $M$ is an irreducible 3-manifold, then $\pi_2(M)$ is trivial.
\end{corollary}

Next we need some results about the types of irreducible 3-manifolds
that can be the link of a complex surface singularity.

\begin{definition}
A Seifert fibration is a triple $(M,f,\Sigma)$ where $M$ is a
3-manifold, $\Sigma$ is a surface and $\morf{f}{M}{\Sigma}$ is
smooth map with the following properties:
\begin{enumerate}
\item For every $x\in \Sigma$ there is a disc neighbourhood $D$ of $x$
  such that $f^{-1}(D)\cong S^1\times D$.
\item If $(t_1,rt_2)$ are polar coordinates on $f^{-1}(D)\cong S^1\times D$
  then thre exist $p,q\in \Z$ depending on $x$ with $\gcd(p,q)=1$ and
  $p\neq 0$ such that $f(t_1,rt_2)=rt_1^qt_3^p$.
\end{enumerate}
A 3-manifold that posses a Siefert fibration is called a Seifert
fibred 3-manifold.
\end{definition}Examples of Seifert fibred manifolds are $S^3$, with
the hop-fibration, circle-bundles in general, lens spaces whit a
Seifert fibration coming as a quotient of the Hopf-fibration of $S^3$.

The definition of Seifert fibred manifolds does not assume that
$\Sigma$ is an oriented surface, but in our case $\Sigma$ will
correspond to complex curves and therefore has to be oriented. So we will assume
the $\Sigma$ is oriented. Except for the lens spaces (here we count
$S^3$ and $S^1\times S^2$) then any homeomorphism of a Seifert fibred
manifold preserves the Seifert fibration. (Here again is used that the
base surface is oriented, if not there are a few extra exceptions)

A fibre of the Siefert fibration is called singular if the value of
$p$ not equal to $+1$. Since $M$ is compact then the number of
singular fibres is finite. It follows that any Siefert fibration can
then be encoded by the Seifert data which can be put in a normal form
as $\bigl(g,(1,b),(p_1,q_1),\dots,(p_1,q_n)\bigr)$. Here $g$ is the
genus of $\Sigma$, $(p_1,q_i)$ are invariants of the singular fibres
with $\gcd(p_1,q_i)=1$ and $1\leq q_i<p_i$ ($p_i$ is the same as in
the definition, but $q_i$ is different) and $b$ is the euler number of
the fibration of the compliment of the singular fibres. For a lens
space $L(p,q)$ we simple associate the Seifert invariants
$\bigl(0,(1,0),(p,q)\bigr)$. There are different conventions of how to
present the Seifert invariants, and we will follow the one from
\cite{Neu07}, see also \cite{brandies} for a detailed construction but be aware they use to oposite convention for the Seifert invariants. We will use the Seifert invariants in section
\ref{topologyandaction} when we want to see the precise topology of
the images of finite holomorphic maps.

\begin{definition}
Let $M$ be a 3-manifold, then an embedded surface
$\Sigma\in M$  is called \emph{compressible}, if on of the following holds:
\begin{enumerate}
\item $\Sigma$ is s sphere and it bounds a ball.
\item There is an immersed disc $D\subset M$ such that $D\cap \Sigma
  = \partial D$ and $\partial D$ is not contractible in $\Sigma$.
\end{enumerate}$\Sigma$ is called \emph{incompressible} if it is not
compressible.
\end{definition} Incompressible surfaces have the following property
which will be used later, for proof of this see for example
\cite{hempel} Corollary 6.2.
\begin{proposition}\label{injectionofincompressiblesurfaces}
Let $\Sigma\subset M$ be an oriented incompressible surface, then
$\ker(\pi_1(\Sigma) \mapsto \pi_1(M))$ is trivial. 
\end{proposition} The main reason for introducing incompressible
surface is the following decomposition theorem for irreducible three
manifolds by Jaco and Salen \cite{jacosalen} and independetly Johannson
\cite{johannson}. 

\begin{theorem}[JSJ-Decomposition]
Let $M$ be a irreducible 3-manifold, then there exist a maximal
sets of pairwise
non isotopic incompressible tori $T_1,\dots T_k$, which is unique up to
isotopy and reordering, such that $M\setminus \bigl(T_1\cup\dots\cup
T_k\bigr)$ is a collection of Seifert fibred (with solid tori consisting of fibres removed) or atoroidal manifolds.
\end{theorem}A manifold $M$ is atoriodal if every embedded tori either
bounds a solid tori or is isotopic to a boundary component of $M$.

\begin{definition}
An irreducible 3-manifold is called a \emph{graph manifold} if
every piece of its JSJ-decomposition is Seifert fibred.
\end{definition} There is a class of 3-manifolds called
\emph{plumbed 3-manifolds} which by results of Waldhausen is the
same as graph manifolds \cite{waldhausen}. It follows from classical result in
resolution of surface singularities \cite{duval} that the link of a
normal surface singularity is a plumbed 3-manifold given by the
dual resolution graph, and hence a
graph manifold. We will describe the plumbing construction in more
detail in section \ref{topologyandaction}, where we use it to study the
links of the images of finite maps by giving their dual resolution graphs.

\section{Maps from $S^3$}\label{quotientsing}

To prove the main Theorem we want to reduce to topological arguments,
and we do that by looking at the map induced by
$\morf{f}{(\C^2,0)}{(\C,0)}$ between the ``links'' of $(\C^2,0)$ and
$(\im f, 0)$. So one needs that $f$ in fact induces such maps. This
seems to be the case in general for finite maps, not just for the case
the domain is $\C^2$. The result we could in the literature
is part $(i)$ of Theorem 3 by Fukuda in the article
\cite{fukuda}, which shows it under a generality condition on $f$,
which we are not sure contains all finite map germs. He also does not
state to prove it in the case of $(\C^2,0)$, but that is just because the
Poincare Conjecture was not proven in that dimension when he wrote his article.

We will give a proof that works for all finite maps, but our proof
only works in the case of 
maps from $(\C^2,0)$ since it will rely on 3-dimensional topology. We
start with a very general lemma which is probably know to experts and
seem to be implicitly used in Fukuda's result.


\begin{lemma}\label{bounderyofsmootcones}
Let $L$ be a finite $n$-dimensional CW-complex and $X:=L\times
[0,1]/(x,0)\sim (y,0)$ be the cone over $L$. Suppose that $X\setminus
L$ is and topological manifold, then $L$ is a homotopy sphere. In
particularly if $L$ is a manifold then it is homeomorphic to $S^n$.
\end{lemma}

\begin{proof}
Let $x^*\in X$ be the cone point and $B$ a open ball neighbourhood of $x^*$ in
$X$. Let $L^*:=X\setminus\{x^*\}$ and $B^*:=B\setminus\{x^*\}$. Then $X=
L^*\cup B$ and $L^*\cap B = B^*$. Now $X$ nad $B$ are contractible so it
follows from the Meyer-Vietoris sequence that the inclusion of $B^*$
in $L^*$ induces the isomorphism $H_i(L^*)\cong
H_i(B^*)$ for all $i$. But $L^*$ is homotopic to $L$ and $B^*$ is homotopic to
$S^n$. So $H_i(L)\cong H_i(S^n)$ for all $i$. Since $B^*$ is simply
connected it follows from Seifert-Van Kampen's Theorem that $\pi_1(X)
\cong \pi_1(B)*\pi_1(L^*) \cong \pi_1(L)$, here we use $n>2$ but it is
easy to modify the argument in the case $n=2$. But $X$ is
contractible so $\pi_1(L)$ is trivial. Since $L$ is a finite CW-complex
it follows form Whitehead's Theorem that it is homotopic to
$S^n$. Then if $L$ is a manifold it is homeomorphic to $S^n$ by the
Poincare Hypothesis.
\end{proof} In our case we do not actually need the Poincare
Conjecture to draw the last conclusion. As we will see later, graph
manifolds with trivial fundamental group being homeomorphic to $S^3$
was know long before the proof of the Poincare Conjecture in three
dimensions by Perelman.

\begin{lemma}\label{linktolink}
Let $F:(\C^2,0)\to (\C^n,0)$ be a germ of a finite holomorphic
map and $L_\epsilon = S^{2n-1}_\epsilon \cap \im F$ be the link of
$(\im F,0)$.  Then there exist an $\epsilon_0$ such that for
$0<\epsilon<\epsilon_0$ the preimage
$\widetilde{S}:=F^{-1}(L_\epsilon)$ is homeomorphic to $S^3$.
\end{lemma}

\begin{proof}
We chose a proper representative of the $F$, which we denote by
$f$, such that $f^{-1}(0)=\{0\}$. Which is possible since $F$ is
the germ of a finite map. Then we can
choose $B_{\epsilon_0}\subset \C^n$ such that the set of critical values
$\Sigma(f)$ of $f$ is a finite union of irreducible complex curves
$C_1$, \dots $C_k$ in $X_{\epsilon_0} := B_{\epsilon_0} \cap \im
f$. Let $l_\epsilon := S_\epsilon^{2n-1}\cap \im f$ be the link of
the image of the representative $f$. Then it is clear that
$l_\epsilon$ is homeomorphic to $L_\epsilon$ for $\epsilon < \epsilon_0$. 

Moreover, since we are working with germ we can choose our
representatives such that each irreducible component $C_i\subset \Sigma(f)$ is
homeomorphic to a cone over its link $K_i$, which is a knot in
$l_\epsilon$. We can
also choose the representative so small, such that $\Sigma(f)$
intersects $S^{2n-1}_\epsilon$ transversely for $\epsilon <\epsilon_0$. Then
the conical structure theorem implies that $X_{\epsilon_0}\setminus
\Sigma(f)$ is homeomorphic to 
$\hat{L}\times (0,\epsilon_0)$, where $\hat{L} = l_\epsilon\setminus
\bigcup_{i=1}^k K_i$. This is of course only true if $0$ is a critical
value, but if 
$0$ is not a critical value then $f$ is an embedding and the result is
trivial.

This implies that $f\vert_{f^{-1}\bigl(X_{\epsilon_0}\setminus
  \Sigma(f)\bigr)} : f^{-1}\bigl(X_{\epsilon_0}\setminus \Sigma(f)\bigr) \to
\hat{L}\times (0,\epsilon_0)$ is a covering map. So $
f^{-1}\bigl(X_{\epsilon_0}\setminus \Sigma(f)\bigr)$ is homeomorphic to
$\widetilde{L}\times (0,\epsilon_0)$ where $\widetilde{L}$ is a finite
cover of $\hat{L}$.

Now $\hat{L}$ is a graph manifold with a finite number of knots
removed, and by the transversality of $\Sigma(f)$ and
$S^{2n-1}_\epsilon$ we have that these knots are fibres of Seifert
fibred pieces of $\hat{L}$. Since $\widetilde{L}$ is a finite cover
of $\hat{L}$ we also have that it is a graph manifold with a finite
number of fibres removed from its Seifert fibred pieces.
 
Let $D_i$ be a conical neighbourhood of $C_i\setminus \{0\}$. That is
$D_i$ is homeomorphic to a tubular neighbourhood of the image of
$C_i\setminus \{0\}$ under the homeomorphism $X_{\epsilon_0} \setminus
\{0\}\approx l_\epsilon\times (0,\epsilon_0)$ and $D_i\cap D_j = \emptyset$ if
$i\neq j$. 

Topological $D_i$ is just a cylinder times a disc with the
central cylinder being the set of critical values. We can also choose
$D_i$ such that $D_i\cap S^{2n-1}_\epsilon$ is a solid torus
neighbourhood homeomorphic to a neighbourhood $T_i$ of the
singular fibre $K_i$ in a Seifert fibred piece of $l_\epsilon$,
one of the fibres removed in $\hat{L}$. This means that
$f\vert_{f^{-1}\bigl(D_i\bigr)} : f^{-1}\bigl(D_i\bigr) \to
D_i$ is a branched cover branched along the central cylinder. So we get
that its connected components are also cylinders times disks, 
and in particular that $f^{-1}\bigl(T_i\bigr)$ is
a union of solid tori neighbourhoods of the components of
$f^{-1}\bigl(K_i\bigr)$.  
  
Since the fibrations $f\vert_{\widetilde{L}}$ and
$f\vert_{f^{-1}(T_i)}$ comes from the restrictions they agree on
$\widetilde{L}\cap f^{-1}(T_i)$ for each $i$. Hence we get that
$\widetilde{S}= f^{-1}(L) =\widetilde{L}\cup \bigcup_{i=1}^k
  f^{-1}(T_i)$ is a graph manifold. Technically the structure we get
  by the gluing might just be the structure of a \emph{graph
    orbifold} (see \cite{myarticle}), but every graph orbifold have an
  underlying space that is a graph manifold, so the topological
  argument that follows is not a problem. 

These considerations implies that $f^{-1}(X_{\epsilon_0}\setminus
\{0\})$ is homeomorphic to $\widetilde{S}\times (0,\epsilon_0)$. Now
choosing a closed ball $\overline{B}_\delta\subset B_{\epsilon_0}$, then
$f^{-1}\bigl((\overline{B}_\delta\cap \im(F))\setminus \{0\}\bigr)$ is
homeomorphic to $\widetilde{S}\times (0,\delta]$. Since we have
chosen $f$ proper and $f^{-1}(0)=\{0\}$ it follows that
$f^{-1}(\overline{B}_\delta\cap \im(f))$ is the cone over
$\widetilde{S}$. Since $f^{-1}(\overline{B}_\delta\cap
\im(F))\setminus \widetilde{S}$ is an open subset of $\C^2$ the result
follows from Lemma \ref{bounderyofsmootcones}

\end{proof}
Here we constructed the cone structure on the preimage by topological
methods specific to three dimensional manifolds. To prove this theorem
in any dimension the best way would probably follow what Fukuda did
and construct vector fields that can be integrated to give the cone structure.

\begin{proof}[Proof of Theorem \ref{mainthm}]

\quad

\paragraph{\eqref{equivquitoinetsing} $\Rightarrow$ \eqref{equivimageofmap}}
First \eqref{equivquitoinetsing} implies that we have the quotient map
$Q:(\C^2,0)\to (X,0)$ and since the group is finite, this is a finite
map. If $i:(X,0)\to (\C^n,0)$ for some $n$ is an embedding, then
$F=i\circ Q$ is the germ of a finite holomorphic map $(\C^2,0)\to
(\C^n,0)$ with $(X,0)$ as its image. 

\paragraph{\eqref{equivimageofmap} $\Rightarrow$  \eqref{equivtrivialfundgroup}}
Assume that $(X,0)\subset (\C^n,0)$ is the image of the germ of a
finite map $F:(\C^2,0) \to (\C^n,0)$.

Let $L$ be the link of $(X,0)$ chosen as in Lemma
\ref{linktolink}. Hence we have a 
surjective map $\morf{f}{S^3}{L}$ gotten by restricting a
representative of $F$ to $\widetilde{S}$ and identifying
$\widetilde{S}$ with $S^3$. Since $(X,0)$ is normal, then $L$ is an
oriented graph manifold. Hence we start our
analysis just considering maps from $S^3$ to oriented graph manifolds with a negative definite plumbing.


Let us consider two cases. The first is the case where $\pi_1(L)$ is
infinite. Let $\widetilde{L}$ be the universal cover of $L$, then
$\pi_i(\widetilde{L}) =\pi_i(L)$ for $i> 1$ and
$\pi_1(\widetilde{L})=\{0\}$. By Corollary \ref{spheretheorem}
$\pi_2(L)=\{0\}$ since $L$ is irreducible. Then by the
Hurewicz Theorem it follows that $\pi_3(L)= \pi_3(\widetilde{L}) =
H_3(\widetilde{L})$, and $H_3(\widetilde{L})= \{0\}$ since
$\widetilde{L}$ is non compact because $\pi_1(L)$ is infinite. Notice
that continuing this by induction we can prove that $L$ is an
Eilenberg-MacLane space $K(\pi_1,1)$. 

Since $\pi_3(L)=\{0\}$ we have that $f$ is homotopic to a constant map
since it is a map from $S^3$ to $L$. Because the topological degree is
a homotopy invariant it follows that $\deg_t(f) =0 $.

Let $0<\epsilon' < \epsilon$ then $\morf{f'}{F^{-1}(S_{\epsilon'}^{2n-1})}{
S_{\epsilon'}^{3}}$ by the above arguments is also a map from a
sphere to a link of $(X,0)$ and hence $\deg_t(f')=0$. Let $A\subset \C^n$
be the annulus bound by $S_\epsilon^{2n-1}$ and $S_{\epsilon'}^{2n-1}$
and $M:=F^{-1}(A)\in\C^2$. Notice that $M$ is the region in $\C^2$
bound by $F^{-1}(S_\epsilon^{2n-1})$ and
$F^{-1}(S_{\epsilon'}^{2n-1})$, and therefore a compact manifold with
boundary whose interior is a complex manifold. (The boundary might
have corneres, but they can canonically be smooted with out changing
any topology and only modifying small neighbourhoods of the corners.) 

Consider $\morf{F':=F\vert_M}{M}{A\cap X}$. Notice that $A\cap X$
is a complex manifold with boundaries, since $X$ is normal and
hence has an isolated singularity, so we can choose our
$\epsilon$ such that $B_\epsilon\cap (X\setminus \{0\})$ is
a complex manifold. $F'$ is a smooth map, holomorphic
on the interior of $M$, and $\deg_t F' = \deg_t F'\vert_{\partial M} = \deg_t
f =\deg_t f'= 0$ by the above consideration and 13.2.1 Theorem of
\cite{DubrovinFomenkoNovikovModernGeometryII}. Since $F'$ is smooth,
then one can calculate its topological degree as the sums of the signs
of the Jacobian determinants at the preimages of a regular value in
the interior of $A$. But $F'$ is holomorphic in the interior and hence
the sign of the Jacobian determinants are always positive. So $\deg_t
F' = 0$ implies that the preimage of a regular value is empty. This is
absurd since $F'$ is surjective by construction, and the set
of critical values have measure $0$ in $A$ by Sard's Theorem.

So only the case where $\pi_1(L)$ is finite can happen, and we have
proved \eqref{equivtrivialfundgroup}.

\paragraph{ \eqref{equivtrivialfundgroup}  $\Rightarrow$ \eqref{equivquitoinetsing}}
By Prill's Theorem 3 \cite{PrillLocalClasificationsOfQuotients} it
follows that the image of $F$ is a quotient of $(\C^2,0)$ by a group
of biholomorphisms of $(\C^2,0)$ isomorphic to $\pi_1(L)$. Since $L$
is the link of a of a singularity it is a graph manifold, and by Proposition
\ref{injectionofincompressiblesurfaces} any incompressible tori would
give an $\Z^2$ summand of $\pi_1(L)$. So $L$ can only have a single
componenet in its JSJ-decomposition since $\pi_1(L)$ is finite, and
hence $L$ is Seifert fibred. It then follows form the
work of Seifert \cite{seifertTopologyofThreeDimensionalFibreSpaces}
and Hopf \cite{hopfOnTheClifford-KleinSpaceProblem} that $\pi_1(L)$ is
a a subgroup of $U(2)$ acting linearly.

We could instead refer to Brieskorn's classification of normal
surface singularities with finite fundamental group
\cite{BrieskornRationalSingularitiesOfComplexSurfaces} for this direction.
\end{proof}

As we saw in the proof the link of the image  $L$ is a quotient of
$S^3$ by a finite subgroup of $SO(4)$ acting linearly. This implies
that they are Seifert fibred. The topology of these manifolds and
therefore their fundamental groups are well know. They were first classified by
Hopf \cite{hopfOnTheClifford-KleinSpaceProblem}  and Seifert
\cite{seifertTopologyofThreeDimensionalFibreSpaces}, see also Scott
\cite{scott} and Orlik \cite{orlik} for presentation of this. In
\cite{BrieskornRationalSingularitiesOfComplexSurfaces} Brieskorn also
derives these results while classifying singularities with finite
fundamental group using more algebraic methods. 

The topological types comes in several infinite classes which we will
discuss individually in more detail in the Section \ref{topologyandaction}. One thing to
be aware of is that not all the 
actions of the subgroups of $SO(4)$ give actions on $\C^2$, so some
topologies presented in \cite{scott} and \cite{orlik} the are not
singularity links. We will of course explain which are singularity
links, but even in these case it might be possible that some action of
their fundamental group, different from the ones we will present later, does not
act on $\C^2$.

In all cases it will follow that the singularities are rational from
using Laufer's criteria \cite{lauferrational} on the minimal dual
resolution graphs we present in the section \ref{topologyandaction}. In
particular the answer to Pe\~nafort's original question is; only the
simple singularities $A_n$, $D_n$, $E_6$, $E_7$ and $E_8$ can be image
of a germ of a finite map from $(\C^2,0)$ to $(\C^3,0)$, since these
are the only hypersurface singularities that are rational.

The dual resolution graphs will also show using the work of Laufer
\cite{laufertauttwodsing} that all these singularities are taut, that
is each topology supports a unique algebraic structure. So we will in
Section \ref{topologyandaction} focus on the topological properties.

\section{The Non Normal Case}\label{nonnormal}

\begin{example}\label{cyclicquotientnromalization}
Consider the singularity $(X,0)\subset \C^3$ given by $xy^{n-q}+z^n=0$
with $\gcd(n,q)=1$. Then $F(s,t)=(s^n,t^q,st)$ is a parametrization of
$(X,0)$. But if $q\neq n-1$ then $(X,0)$ is not a simple singularity,
it is easy to see since its multiplicity is not $2$. This does not
contradict the conclusion of the last Section, since
$(X,0)$ is not
normal. In fact it does not have an isolated singularity, since the
$x$-axis is the singular set.
\end{example}

If $F:(\C^2,0)\to (X,0)$ is a finite holomorphic map, then
by the universal property of normalization, since $\C^2$
is normal, there exist a commutative diagram
\[
\begin{tikzcd}
(\C^2,0) \arrow[r, "{\widetilde{F}}"] \arrow[rd, "{F}"]   
 & (\widetilde{X},0)\, \arrow[d, "{G}"]
\\ 
& (X,0),
\end{tikzcd}
\]
where $\morf{G}{(\widetilde{X},0)}{(X,0)}$ is the normalization. 
Because $F$ is finite it is clear that $\widetilde{F}$ is also finite
and since $F^{-1}(0) =\{0\}$ the same is true for $\widetilde{F}$. We
can then use Section
\ref{topologyandaction} to conclude that $(\widetilde{X},0)$ is one of
the singularities described there. 

If $(X,0)\subset (\C^3,0)$ then $(\widetilde{X},0)$  wont necessarily
be a hypersurface, and in Example
\ref{cyclicquotientnromalization} the normalization is a 
cyclic quotient singularity with link the lens space $L(n,q)$, which is
not a hypersurface unless $q=n-1$.

If $(X,0)\subset (\C^n,0)$ is locally irreducible then the normalization map
$\morf{G}{(\widetilde{X},0)}{(X,0)}$ is a homeomorphism see Section
12.2.6 of \cite{nemethinormalsurfacesing}. In this
case the topologies described in Section \ref{topologyandaction} are
also the possible topologies of $(X,0)$. But be aware that links might
not be smoothly embedded in $\C^n$. In particular, if $n=3$ then
$(X,0)$ not being normal is the same as having non isolated
singularities. So normalizations are only needed in this case if the
singularities are non isolated as in Example \ref{cyclicquotientnromalization}.

If $(X,0)$ is not locally irreducible, then we can have singularities
like the Whitney umbrella $x^2-y^2z=0$ which is the image of
$(st,s,t^2)$. Notice that the map is generically one to one, 
so the map restricted to $S^3$ is also generically
one to one and hence the normalization is smooth. But the link is not
$S^3$ since it has double points.

To conclude. If one have a singularity $(X,0)\subset (\C^n,0)$ and want
to see if it is the image of a finite map, then the first
step is to normalize $(X,0)$. The second step is to identify if
$(\widetilde{X},0)$ is a quotient singularity, which can be done by
looking at the topology of the link. To construct a map with $(X,0)$
as the image, one needs to construct a map to $(\widetilde{X},0)$, which
can be done using the quotient construction as seen in
Section \ref{topologyandaction}, and then
compose with the normalization map. Which is exactly the procedure of
Algorithm \ref{nonnomralalgo}

\section{Topology, Action and Maps}\label{topologyandaction}

In this section we will do two things, first give a complete
description of the topologies that can appear as the image of a germ
of a finite map. This will be done by describing the Seifert
invariants of the link, givening the plumbing diagram of the link/dual
resolution graph of the singularity, and finally describing the
fundamental group of the link. Notice as mentioned in the end of the
Section \ref{quotientsing}, this topological description also completely determines
the algebraic structure of the images, since these singularities are
taut.

The second thing we do in this section is to give the procedure for
how given one of the topologies one can find equations that give a
singularity with the given topology and a map from $\C^2$ to $\C^n$
that have this singularity as its image. We will completely do the
cases of the simple singularities, which are the cases $n=3$ and do
some examples with higher embedding dimension.

It follows form the classification of Siefert fibred manifolds that
we can divide the ones with finite fundamental group in two different
classes. The first class is the Lens spaces $L(p,q)$ that have
fundamental group a cyclic group $\Z/p\Z$, and corresponds to to the
cyclic quotient singularities. These manifolds do not have unique
Seifert fibred structure. 

The the manifolds in the other class have a unique Seifert
fibred structure and is fibred over the sphere with at least 3 singular
fibres.  Form Table 4.1 on page 441
of \cite{scott} we have that these topologies are determined by the
following inequalities concerning 
the \emph{orbifold Euler characteristic} $\chi(L)$ and the
\emph{rational Euler number} $e(L)$:
\begin{align*}
&\chi(L) > 0& &\text{and}& &e(L)\neq 0.&
\end{align*}
We will for convenience follow the conventions of \cite{Neu07}
for the normalized 
Seifert invariants $\{0,(1,-b)(p_1,q_1),\dots,(p_k,q_k)\}$ of $L$. Then
the rational euler number and the orbifold euler characteristic are
given by the formulas:
\begin{align*}
&\chi(L) = 2-k+\sum_{i=1}^k\frac{1}{p_i}& &\text{and}& &e(L)=
-b + \sum_{i=1}^k \frac{q_i}{p_i}.&
\end{align*} From the first inequality it follows that there are
only $3$ singular fibres and the Seifert
invariants have to be on the form $\{(1,-b)(2,1)(2,1)(p,q)\}$,
$\{(1,-b)(2,1)(3,q_1)(3,q_2)\}$, $\{(1,-b)(2,1)(3,q_1)(4,q_2)\}$ or
$\{(1,-b)(2,1)(3,q_1)(5,q_2)\}$. Since $L$ is the link of a singularity the
plumbing has to be negative definite, but that is the same as
$e(L)<0$ which can bee seen from the description below. So using the
second equation one gets that $b> 1$.

As mentioned earlier graph manifolds also have a description as
plumbed 3-manifolds, which we will briefly describe. Let $E_1$ and
$E_2$ be two disc bundles over compact surfaces $\Sigma_1$ and
$\Sigma_2$ respectively and let $e_i$ be the euler number of the
bundle $E_i$. If $\Sigma_i$ is orientable then $E_i$ can be seen as
the unit disc bundle in a complex line $L_i$ bundle and $e_i=c_1(L_i)$
is the first Chern class of this bundle. Let $\widetilde{D}_i\subset
\Sigma_i$ be a disc such that $E_i$ trivializes over it. We
have subsets $D_1\times\widetilde{D}_1 \subset E_1$ and
$D_2\times\widetilde{D}_2 \subset E_2$, and by identifying all discs
with the standard dice we can define a gluing of $E_1$ and $E_2$ by
$\morf{\phi}{D_1\times\widetilde{D}_1}{D_2\times\widetilde{D}_2}$
where $\phi(x,y)=(y,x)$. That is we identify the fibres in one
fibration with the disc in the base of the other fibration.

If one have a finite collection of disc bundles $E_1$,\dots,$E_k$, then
one can continue to do this process provided that we never glue along
subsets that have already been used for the gluing once before. The
result is a \emph{plumbed four manifold}, and we can store the data of
the construction in a \emph{plumbing graph} which have a vertex $v_i$ for
each $E_i$ an edge between $v_i$ and $v_j$ if the corresponding disc bundles are glued. We
decorate the vertices of the graph with the euler number of the
corresponding disc bundle and genus of the base surface of the
fibration (usually we do not write the genus if it is $0$). 
Then the plumbing graph give a way to construct the plumbed
four manifold which is unique up to homeomorphism. 

\emph{A plumbed 3-manifold} is the boundary of a plumbed four
manifold after smoothing the corners. A plumbing graph of a plumbed
four manifold is also a plumbing graph of its 3-manifold
boundary. But in this case it is not unique, there exist a series of
plumbing moves that can be used to change the plumbed four manifold
without changing its boundary. But Neumann proves in \cite{plumbing}
that there exist a unique normal form plumbing diagram. Which in our
case can be described as any vertex $v_i$ with only one or two neighbors
and a genus weight of $0$ have euler number $e_i\leq -2$.

Siefert fibred manifolds are plumbed manifolds, their plumbing graphs are star
shaped with a central vertex that can have a base with genus different
from $0$ and with one leg (or bamboo) for each singular fibre, along
which all the base surfaces are spheres. If the
Seifert invariants are  $\{g,(1,-b)(p_1,q_1),\dots,(p_k,q_k)\}$ then
the plumbing graph is:
\begin{center}
\begin{tikzpicture}
\draw[thin](0,0)--(2,2);
\draw[thin](0,0)--(2,1);
\draw[thin](0,0)--(2,-2);

\draw[thin](2,2)--(4,2);
\draw[thin](2,1)--(4,1);
\draw[thin](2,-2)--(4,-2);

\draw[thin](4,2)--(5,2);
\draw[thin](4,1)--(5,1);
\draw[thin](4,-2)--(5,-2);

\draw[dashed](5,2)--(7,2);
\draw[dashed](5,1)--(7,1);
\draw[dashed](5,-2)--(7,-2);

\draw[thin](7,2)--(8,2);
\draw[thin](7,1)--(8,1);
\draw[thin](7,-2)--(8,-2);

\draw[fill=white](0,0)circle(2pt);
\draw[fill=white](2,2)circle(2pt);
\draw[fill=white](2,1)circle(2pt);
\draw[fill=white](2,-2)circle(2pt);

\draw[fill=white](4,2)circle(2pt);
\draw[fill=white](4,1)circle(2pt);
\draw[fill=white](4,-2)circle(2pt);

\draw[fill=white](8,2)circle(2pt);
\draw[fill=white](8,1)circle(2pt);
\draw[fill=white](8,-2)circle(2pt);

\node(a)at(0,0.45){$-b$};
\node(a)at(2,2.35){$e_{11}$};
\node(a)at(4,2.35){$e_{12}$};
\node(a)at(8,2.35){$e_{1m_1}$};
\node(a)at(2,1.35){$e_{21}$};
\node(a)at(4,1.35){$e_{22}$};
\node(a)at(8,1.35){$e_{2m_2}$};
\node(a)at(2,-1.65){$e_{k1}$};
\node(a)at(4,-1.65){$e_{k2}$};
\node(a)at(8,-1.65){$e_{km_1}$};

\node(a)at(2,-.2){$\vdots$};
\node(a)at(4,-.2){$\vdots$};
\node(a)at(8,-.2){$\vdots$};

\node(a)at(0,-.5){$[g]$};

\end{tikzpicture}
\end{center}Where 
\begin{align*}
p_i/q_i =-e_{i1}-\cfrac{1}{-e_{i2}-\cfrac{1}{\ddots -\cfrac{1}{-e_{im_i}}}}.
\end{align*}Notice that since $e_{ij}\leq -2$ they are uniquely
determined by $p_1/q_i$. For this construction see Section 5 of
\cite{plumbing} but be aware that Neumann there uses a different
convention for the Seifert invariants, so the formulas are a little
different in the one we use.

Below we list the different minimal plumbing graphs for the given Seifert invariants, which in the case of lens spaces is just the numbers $(p,q)$. Afterwards we run Laufer's algorithm \cite{lauferrational} to calculate the minimal cycle $Z_{min}$. The algorithm will in all cases show that $(X,0)$ is a rational singularity and hence $Z_{min}=Z_{max}$ the maximal ideal cycle. Since the multiplicity of $(X,0)$ is equal to $-Z_{max}^2$ we calculate its value from the Seifert invariants. This is later useful, since $(X,0)$ being rational implies that Abyankar's inequality is a equality, and therefore that the minimal embedding dimension is equal to $\mul(X,0) + 1$. We also describe which group acts in each case.

To construct actual maps with a given singularity $(X,0)$ as its
image, we will construct $(X,0)$ as the the quotients of a given
action on $(\C^2,0)$ obtaining the quotient maps in the
process. To construct a geometric quotient of an action by and
algebraic group $G$ on
$(\C^2,0)$ one can do it in the following way using the geometric
invariant theory of Mumford \cite{mumfordgit}: The local algebra of
the quotient is $\C[u,v,]^G$ the ring of invariants of the
action. To find this ring we first find a generating set of invariant
polynomials $p_1$, $p_2,$\dots, $p_k$. Here we can use that we first
found the embedding dimension, so once we have as many
algebraically independent polynomials as the embedding dimension we can
stop looking for more. Then $\C[u,v]^G\cong \C[x_1,\dots,x_k]/I$ where
$I$ is the ideal generated by all the relations arising form setting
$x_i=p_i(u,v)$. Here we unfortunately do not have an upper bound of the
number of algebraically independent relations except it is higher than
$k-2$ unless the quotient is a simple singularity. But $I$ is the
kernal of a map and hence using computer algebra we can find a minimal
set of generators as we will see later. The map
$F:(\C^2,0)\to (\C^k,0)$ with image the given singularity is then
$F(u,v)=\bigl(p_1(u,v),p_2(u,v),\dots, p_k(u,v)\bigr)$.

\subsection{Cyclic Quotients}
The first family of links is the lens spaces $L(p,q)$ with
$\gcd(p,q)=1$. In this case $\pi_1\bigl(L(p,q)\bigr)= \Z/p\Z$ and it
acts on $S^3\subset \C^2$ by 
multiplication with the matrix $\begin{psmallmatrix} e^{2\pi i/p} & 0 \\ 0
  & e^{2\pi iq/p} \end{psmallmatrix}$. This action clearly extends over all
of $\C^2$ and the image of $F$ is the cyclic quotient singularity
given by this action. Notice here that $\Z/p\Z$ may have many
different actions giving possibly different quotients, but non of
these are singularity links unless they are diffeomorphic to one of the form above. We will call this action a  $(p,q)$ action and
remember that if $qq'\equiv 1 \mod p$ then the $(p,q)$ and $(p,q')$
actions give homeomorphic and hence isomorphic quotients since as we mentioned before the singularities are taut.

The plumbing diagrams in this case are 
\begin{center}
\begin{tikzpicture}
\draw[thin](0,0)--(1,0);
\draw[thin](1,0)--(2,0);
\draw[fill=white](0,0)circle(2pt);
\draw[fill=white](1,0)circle(2pt);

\draw[fill=white](4,0)circle(2pt);
\draw[dashed](2,0)--(3,0);
\draw[thin](3,0)--(4,0);

\draw[thin](4,0)--(5,0);

\draw[fill=white](4,0)circle(2pt);

\draw[fill=white](5,0)circle(2pt);

\node(a)at(0,0.35){$-a_1$};
\node(a)at(1,0.35){$-a_2$};
\node(a)at(4,0.35){$-a_{k-1}$}; 
\node(a)at(5,0.35){$-a_k$};

\end{tikzpicture}
\end{center}where $p/q=a_1-\cfrac{1}{a_2-\cfrac{1}{\dots
    -\frac{1}{a_k}}}$. 

Notice that if $p=1$ then $q=0$ and $p/q$ should
be interpreted as $1$, so then the plumbing diagram is just a single
vertex with weight $-1$ which is $S^3$. That is this is the case of
the image being a smooth point. We run Laufer's algorithm as explained before and get that $(X,0)$ is rational and $\mul(\im F) = -2(k-1) + \sum a_i$.


\subsubsection{Maps in the case of the Cyclic  groups}\label{mapsforcyclicgroups}

To construct maps we start by finding invariant polynomials. First
notice that $u^p$ and $v^p$ are invariant. It is clear that no lower
power of $u$ is invariant and since $\gcd(p,q)=1$ we also have that no
lower power of $v$ is invariant. Next we look at monomial of the form
$u^av^b$ with $a,b> 0$. Such a monomial is invariant if and only if
$a+qb\equiv 0 \mod p$ or equivalently $a\equiv -bq \mod p$. If we use
$q'$ instead of $q$ it is $b\equiv -aq' \mod p$ so we notice that we
get the same pairs. This will give us $p+1$ different invariants
monomials, and we can get a minimal generating set among them doing
the following procedure: 

We let $b$ run through the numbers from $0$
to $p$ and calculate $a$ the following way. Starting from $b=0$ we
simply let $a= p- bq$, this works until $p- (k_1+1)q<0$ and we will have
gotten pairs $(a_0,0)$, \dots, $(a_{k_1},k_1)$ which we consider as vectors
in $\N^2$. We get $a_{k_1+1} = 2p-(k_1+1)q$, but now we need to be
careful. If the pair $(a_{k_1+1},k_1+1)$ is in the semigroup
generated by $(a_0,0)$, \dots, $(a_{k_1},k_1)$ in $\N^2$, then the
monomial $u^{a_{k_1+1}}v^{k_1+1}$ is a product of some of the
monomials $u^{a_0}$, \dots, $u^{a_{k_1}}v^{k_1}$. If that is the case
we discard the pair $(a_{k_1+1},k_1+1)$, if not we add the pair to the
list. We then
continue with $a_{k_1+2} = 2p-(k_1+2)q$ taking care to check if we
need to add it or not. When $a_{k_2+1} = 2p-(k_2+1)q<0$ we move on to
$a_{k_2+1} = 3p-(k_2+1)q$. This process will continue until $a_p=
qp-pq = 0$ giving the last monomial $v^p$. Since we were careful to
discard any monomial that was a multiple of previous obtained
monomial, this will result in a minimal list of generating invariant
monomials.

\begin{example}\label{exp=5q=3}
Let $p=5$ and $q=2$. Then the first pair is as always $(5,0)$. The next
is $(3,1)$ and we also get $(1,2)$ before $5-3\cdot 2<0$. Following
the algorithm described above the next pair should be $(a_3,3)$ where
$a_3 = 2\cdot 5-3\cdot 2 = 4$, but notice that $(4,3)= (3,1) + (1,2)$
so this pair does not give a new generator. The same with $(2,4)$
since $(2,4) = 2(1,2)$. The last pair to add is of course $(0,5)$.
So a generating set of monomials is $u^5,u^3v,uv^2,v^5$.
\end{example}

Finding the relations is the same just consider all monomials given as
products of the generators. It is clear that we do not need any powers
of the monomials that are higher than $p$ and pure powers of $u$ and
$v$ will not appear, so it is a finite set. Then
all the relations will appear as identities of these
monomials. Choosing a minimal algebraically independent set will
define a minimal set of equations for the image.

\begin{example}
Continuing with Example \ref{exp=5q=3} we get that all the possible
products not involving powers of $x=u^5$ and $y=v^5$ we set $z=u^3v$ and
$w=uv^2$. It is clear that $z^5=x^3y$ and $w^5=xy^2$ so we will not
write any monomials that includes $z^5$ or $w^5$ except those two. We
then get:

\begin{align*}
&z^2=u^3v^2,& &z^3=u^{9}v^3,& &z^4=u^{12}v^{4},& &z^5=u^{15}v^{5},&\\
&zw=u^{4}v^{3},& &z^2w=u^{7}v^{4},&
&z^3w=u^{10}v^{5},& &z^4w=u^{13}v^{6},&\\ &zw^2=u^{5}v^{5},&
&z^2w^2=u^{8}v^{8},& &z^3w^2=u^{11}v^{7},&
&z^4w^2=u^{15}v^{8},&\\
&zw^3=u^{6}v^{7},& &z^2w^3=u^{9}v^{8},&
&z^3w^3=u^{12}v^{9},& &z^4w^3=u^{15}v^{10},&\\ &zw^4=u^{7}v^{9},&
&z^2w^4=u^{10}v^{10},& &z^3w^4=u^{13}v^{11},&
&z^4w^4=u^{17}v^{12},&\\ &w^2=u^{2}v^{4},&
&w^3=u^{3}v^{6},& &w^4=u^{4}v^{8},& &w^5=u^{5}v^{10}.&
\end{align*}
We then easily see that the relations are $z^4w^3=x^3y^2$ and
$z^2w^4=x^2y^2$ in addition to $z^5=x^3y$ and $w^5=xy^2$.

Hence the map $F:=(\C^2,0)\to (\C^4,0)$ given by
$F(u,v)=\bigl(u^5,u^3v,uv^2,v^5\bigr)$ will have image the cyclic
quotient singularity $(5,2)$ with equations $z^4w^3=x^3y^2$,
$z^2w^4=x^2y^2$, $z^5=x^3y$ and $w^5=xy^2$.
\end{example}
We finish this case of cyclic quotients by giving the maps for all
groups of order less than $8$ in the table below. We only list one of
the actions $(p,q)$ and $(p,q')$ if $qq'\equiv 1 \mod p$ since these
actions give isomorphic quotients.

\begin{center}
\begin{tabular}{lcc}\toprule
The (p,q) action & Embedding dimension & Map \\ \midrule
$(2,1)$ & $3$ & $(u^2,uv,v^2)$ \\
$(3,1)$ & $4$ & $(u^3,u^2v,uv^2,v^3)$ \\
$(3,2)$ & $3$ & $(u^3,uv,v^3)$ \\
$(4,1)$ & $5$ & $(u^4,u^3v,u^2v^2,uv^3,v^4)$ \\
$(4,3)$ & $3$ & $(u^4,uv,v^4)$ \\ 
$(5,1)$ & $6$ & $(u^5,u^4v,u^3v^2,u^2v^3,uv^4,v^5)$ \\
$(5,2)$ & $4$ & $(u^5,u^3v,uv^2,v^5)$ \\
$(5,4)$ & $3$ & $(u^5,uv,v^5)$ \\
$(6,1)$ & $7$ & $(u^6,u^5v,u^4v^2,u^3v^3,u^2v^4,uv^5,v^6)$ \\
$(6,5)$ & $3$ & $(u^6,uv,v^6)$ \\
$(7,1)$ & $8$ & $(u^7,u^6v,u^5v^2,u^4v^3,u^3v^4,u^2v^5,uv^6,v^7)$ \\
$(7,2)$ & $5$ & $(u^7,u^5v,u^3v^2,uv^3,v^7)$ \\
$(7,3)$ & $4$ & $(u^7,u^4v,uv^2,v^7)$ \\
$(7,6)$ & $3$ & $(u^7,uv,v^7)$ \\ \bottomrule 
\end{tabular}
\end{center}

\subsection{Binary Polyhedral Quotiones}

As we will see below, if $\pi_1(L)$ is not a cyclic group it is of the
form $\Z/m\Z\times G$ where $G$ is one of the binary polyhedral groups
(binary dihedral groups $D_{4n}^*$, binary
tetrahedral group $T^*$, binary octahedral group $O^*$, binary
icosahedral group $I^*$) or the groups $D'_{2^{l+2}(2l+1)}$ or
  $T'_{8\cdot3^k}$. As an aside the binary cyclic groups, that is
the double cover of $\Z/p\Z$, also acts on $\C^2$. But its action 
is just the action of multiplication by $\begin{psmallmatrix}
  e^{\pi i/p} & 0 \\ 0 & e^{\pi i/p} \end{psmallmatrix}$, which is one
of the actions by $\Z/2p\Z$ covered under the cyclic quotients.

We will first briefly give a summery of the topology, multiplicity of
the singularity and which group acts from the Seifert invariants in
the different families. Then afterwards first find maps when the fundamental group
is just a binary polyhedral group that acts, and later give examples
of products with the cyclic groups. We will briefly discuss the case
of the groups being $D'_{2^{l+2}(2l+1)}$ or $T'_{8\cdot3^k}$ last.

Notice that maps we construct restrict to the universal covering maps on
the link, but that is not the only possibility for maps in these
cases. $\pi_3(L)\cong \pi_3(S^3)\cong \Z$, and any
class in there can potentially be represented by the restriction of a finite
map from $(\C^2,0)$ to $(\C^n,0)$. But notice that the isomorphism
$\pi_3(L)\cong \pi_3(S^3)$ is from the fibration sequence and is hence
given by the induced map from the quotient map. So any class can
be represented as a composition of a map $S^3\to S^3$ composed with
the quotient map. This means that any map $S^3\to L$ is homotopic to a
map $S^3\to S^3$ composed with the quotient map. 

\subsubsection{Extensions of the  Dihedral Group}
We will consider the each of the cases with $3$ singular fibres
starting with the case of Seifert invariants
$\{(1,-b)(2,1)(2,1)(p,q)\}$. The plumbing diagram is then 
\begin{center}
\begin{tikzpicture}
\draw[thin](0,0)--(1,0);
\draw[thin](1,0)--(2,0);
\draw[thin](1,0)--(1,-1);
\draw[thin](2,0)--(3,0);
\draw[dashed](3,0)--(4,0);
\draw[thin](4,0)--(5,0);
\draw[thin](4,0)--(5,0);

\draw[fill=white](0,0)circle(2pt);
\draw[fill=white](1,0)circle(2pt);
\draw[fill=white](1,-1)circle(2pt);
\draw[fill=white](2,0)circle(2pt);
\draw[fill=white](5,0)circle(2pt);

\node(a)at(0,0.35){$-2$};
\node(a)at(0.60,-1){$-2$};
\node(a)at(1,0.35){$-b$};
\node(a)at(2,0.35){$-a_{1}$}; 
\node(a)at(5,0.35){$-a_k$};

\end{tikzpicture}
\end{center}where $p/q=a_1-\cfrac{1}{a_2-\cfrac{1}{\dots
    -\frac{1}{a_k}}}$. 

Using Laufer's algorithm
we get that if $b> 2$ then the fundamental cycle is reduced and
$\mul(\im F) = -2k +b +\sum a_i $. If $b=2$ then define $l$ such that 
$a_1=a_2=\dots=a_l=a_{l}=2$. Laufer's algorithm then gives $\mul(\im
F) = -2(k-1) +2l + \sum_{i=l+1}^k a_i $ if $l>0$ and $\mul(\im F) =
-2(k-1) + \sum_{i=1}^k a_i $ if $l=0$. In particular we get a
hypersurface only if $b=a_1=\dots=a_k=2$, that is we have a
$D_{k+3}$ singularity.

In this case the fundamental groups of the link depends on the number
$m:= (b-1)p-q$ following Theorem 2 section 6.2 in
\cite{orlik}, the difference between the value of $m$ defined here and
in the reference is that \cite{orlik} uses the other convention
for defining the Seifert invariants. Then there are two different
cases, first if $\gcd(m,2p)=1$ then $\pi_1(L) = \Z/m\Z \times D^*_{4p}$ where
$D^*_{4p}$ is the binary dihedral group. That is $D^*_{4p}$ is the
double cover of the dihedral group $D_{2p}$ under the cover
induced by the cover $S^3\to SO(3)$. 
If $m=2m'$ then $m'$ is even and
$\gcd(m',p)=1$. If $m'=2^km''$ with $m''$ odd, then $\pi_1(\im F) =
\Z/m''\Z \times D'_{2^{k+2}p}$ where $D'_{2^{k+2}p}$ is a subgroup
of $SO(4)$ isomorphic to $\langle x,y \vert x^{2^k}=1,\ y^{p}=1,\
xy^{-1}=y^{-1}x \rangle$. Notice the $D'_{4p}=D^*_{4p}$. 

\subsubsection{Extensions of the Tetrahedral Group}
Next is the case of the Seifert invariants being
$\{(1,-b)(2,1)(3,q_1)(3,q_2)\}$. The plumbing diagram is one of the following graphs
depending on $q_1$ and $q_2$. 
\begin{center}
\begin{tikzpicture}
\draw[thin](0,0)--(1,0);
\draw[thin](1,0)--(2,0);
\draw[thin](1,0)--(1,-1);

\draw[fill=white](0,0)circle(2pt);
\draw[fill=white](1,0)circle(2pt);
\draw[fill=white](1,-1)circle(2pt);
\draw[fill=white](2,0)circle(2pt);

\node(a)at(0,0.35){$-3$};
\node(a)at(0.60,-1){$-2$};
\node(a)at(1,0.35){$-b$};
\node(a)at(2,0.35){$-3$};

\draw[thin](3,0)--(4,0);
\draw[thin](4,0)--(5,0);
\draw[thin](4,0)--(4,-1);
\draw[thin](5,0)--(6,0);

\draw[fill=white](3,0)circle(2pt);
\draw[fill=white](4,0)circle(2pt);
\draw[fill=white](4,-1)circle(2pt);
\draw[fill=white](5,0)circle(2pt);
\draw[fill=white](6,0)circle(2pt);

\node(a)at(3,0.35){$-3$};
\node(a)at(3.60,-1){$-2$};
\node(a)at(4,0.35){$-b$};
\node(a)at(5,0.35){$-2$}; 
\node(a)at(6,0.35){$-2$};

\draw[thin](7,0)--(8,0);
\draw[thin](8,0)--(9,0);
\draw[thin](9,0)--(10,0);
\draw[thin](9,0)--(9,-1);
\draw[thin](10,0)--(11,0);

\draw[fill=white](7,0)circle(2pt);
\draw[fill=white](8,0)circle(2pt);
\draw[fill=white](9,-1)circle(2pt);
\draw[fill=white](9,0)circle(2pt);
\draw[fill=white](10,0)circle(2pt);
\draw[fill=white](11,0)circle(2pt);

\node(a)at(7,0.35){$-2$};
\node(a)at(8.60,-1){$-2$};
\node(a)at(9,0.35){$-b$};
\node(a)at(8,0.35){$-2$}; 
\node(a)at(10,0.35){$-2$}; 
\node(a)at(11,0.35){$-2$};

\node(a)at(1,-1.5){$q_1=q_2=1$};
\node(a)at(4.5,-1.5){$q_1+q_2=3$ };
\node(a)at(9,-1.5){$q_1=q_2=2$};
\end{tikzpicture}
\end{center} 
\begin{center}
\begin{tabular}{lcccc}\toprule
  & $q_1= q_2=1$ & $q_1+q_2=3$ &  $q_1= q_2=2$ \\ \midrule
$\mul(\im F)$ & $2+b$ & $1+b$ & $b$ \\ \bottomrule 
\end{tabular}
\end{center}Once again we see that only the simple singularity $E_6$ ($b=2$ in the last graph) is a
hypersurface singularity.

As in the last case the fundamental group $\pi_1(L)$ depends on a
number defined as $m=
6b-3 -2q_1-2q_2$ according to Theorem 2 section 6.2 in \cite{orlik},
again the difference in formula is because of the different conventions
for the Seifert invariants. If $\gcd(m,12)=1$ then
$\pi_1(L)=\Z/m\Z\times T^*$ where $T^*$ is the binary tetrahedral
group. If $m=3^km'$ and $\gcd(m',12)=1$ then $\pi_1(L)=\Z/m'\Z\times
T'_{8\cdot 3^k}$ where $T'_{8\cdot 3^k}$ is a subgroup of $SO(4)$
isomorphic to $\langle x,y,z \vert x^2=xyxy=y^2,\ zxz^{-1}=y,\
zyz^{-1}=xy, z^{3^k}=1 \rangle$.

\subsubsection{Extensions of the  Octahedral Group}
The third case is when the Seifert invariants are
$\{(1,-b)(2,1)(3,q_1)(4,q_2)\}$ and the plumbing diagram is one of the following
depending on $q_1$ and $q_2$.
\begin{center}
\begin{tikzpicture}
\draw[thin](0,0)--(1,0);
\draw[thin](1,0)--(2,0);
\draw[thin](1,0)--(1,-1);

\draw[fill=white](0,0)circle(2pt);
\draw[fill=white](1,0)circle(2pt);
\draw[fill=white](1,-1)circle(2pt);
\draw[fill=white](2,0)circle(2pt);

\node(a)at(0,0.35){$-3$};
\node(a)at(0.60,-1){$-2$};
\node(a)at(1,0.35){$-b$};
\node(a)at(2,0.35){$-4$};

\draw[thin](4,0)--(5,0);
\draw[thin](5,0)--(6,0);
\draw[thin](5,0)--(5,-1);
\draw[thin](6,0)--(7,0);
\draw[thin](7,0)--(8,0);

\draw[fill=white](4,0)circle(2pt);
\draw[fill=white](5,0)circle(2pt);
\draw[fill=white](5,-1)circle(2pt);
\draw[fill=white](6,0)circle(2pt);
\draw[fill=white](7,0)circle(2pt);
\draw[fill=white](8,0)circle(2pt);

\node(a)at(4,0.35){$-3$};
\node(a)at(4.60,-1){$-2$};
\node(a)at(5,0.35){$-b$};
\node(a)at(6,0.35){$-2$}; 
\node(a)at(7,0.35){$-2$};
\node(a)at(8,0.35){$-2$};

\node(a)at(1,-1.5){$q_1=q_2=1$};
\node(a)at(6,-1.5){$q_1=1$ and $q_2=3$ };

\draw[thin](0,-3)--(1,-3);
\draw[thin](1,-3)--(2,-3);
\draw[thin](2,-3)--(3,-3);
\draw[thin](2,-3)--(2,-4);

\draw[fill=white](0,-3)circle(2pt);
\draw[fill=white](1,-3)circle(2pt);
\draw[fill=white](2,-3)circle(2pt);
\draw[fill=white](2,-4)circle(2pt);
\draw[fill=white](3,-3)circle(2pt);

\node(a)at(0,-2.65){$-2$};
\node(a)at(1.60,-4){$-2$};
\node(a)at(1,-2.65){$-2$};
\node(a)at(2,-2.65){$-b$}; 
\node(a)at(3,-2.65){$-4$};

\draw[thin](4,-3)--(5,-3);
\draw[thin](5,-3)--(6,-3);
\draw[thin](6,-3)--(7,-3);
\draw[thin](6,-3)--(6,-4);
\draw[thin](7,-3)--(8,-3);
\draw[thin](8,-3)--(9,-3);

\draw[fill=white](4,-3)circle(2pt);
\draw[fill=white](5,-3)circle(2pt);
\draw[fill=white](6,-3)circle(2pt);
\draw[fill=white](6,-4)circle(2pt);
\draw[fill=white](7,-3)circle(2pt);
\draw[fill=white](8,-3)circle(2pt);
\draw[fill=white](9,-3)circle(2pt);

\node(a)at(4,-2.65){$-2$};
\node(a)at(5.60,-4){$-2$};
\node(a)at(6,-2.65){$-b$};
\node(a)at(5,-2.65){$-2$}; 
\node(a)at(7,-2.65){$-2$}; 
\node(a)at(8,-2.65){$-2$};
\node(a)at(9,-2.65){$-2$};

\node(a)at(1.5,-4.5){$q_1= 2$ and $q_2=1$};
\node(a)at(6,-4.5){$q_1=2$ and $q_2=3$ };

\end{tikzpicture}
\end{center}
\begin{center}
\begin{tabular}{lcccc}\toprule
  & $q_1= q_2=1$ & $q_1= 1$ and $q_2=3$ & $q_1= 2$ and
 $q_2=1$ & $q_1= 2$ and $q_2=1$ \\ \midrule
$\mul(\im F)$ & $3+b$ & $1+b$ & $2+b$ & $b$ \\ \bottomrule 
\end{tabular}
\end{center}So again we see that only the simple singularity $E_7$ ($b=2$ in the last graph) is
a hypersurface.

The fundamental group $\pi_1(\im F) = \Z/{m\Z} \times O^*$ where
$O^*$ is the binary octahedral group and
$m = 12b-6-4q_1-3q_2$ according to Theorem 2 section 6.2 of
\cite{orlik}. Again be aware that he uses the other convention for the
Seifert invariants, hence the different formula for $m$. 

\subsubsection{Extensions of the Icosahedral Group}
The last case is for the Seifert invariants to be
$\{(1,-b)(2,1)(3,q_1)(5,q_2)\}$. Then depending on $q_1$ and $q_2$ we have
the plumbing graphs: 
\begin{center}
\begin{tikzpicture}
\draw[thin](0,0)--(1,0);
\draw[thin](1,0)--(2,0);
\draw[thin](1,0)--(1,-1);

\draw[fill=white](0,0)circle(2pt);
\draw[fill=white](1,0)circle(2pt);
\draw[fill=white](1,-1)circle(2pt);
\draw[fill=white](2,0)circle(2pt);

\node(a)at(0,0.35){$-3$};
\node(a)at(0.60,-1){$-2$};
\node(a)at(1,0.35){$-b$};
\node(a)at(2,0.35){$-5$};

\draw[thin](6,0)--(7,0);
\draw[thin](8,0)--(8,-1);
\draw[thin](7,0)--(8,0);
\draw[thin](8,0)--(9,0);

\draw[fill=white](6,0)circle(2pt);
\draw[fill=white](8,-1)circle(2pt);
\draw[fill=white](7,0)circle(2pt);
\draw[fill=white](8,0)circle(2pt);
\draw[fill=white](9,0)circle(2pt);

\node(a)at(7.60,-1){$-2$};
\node(a)at(6,0.35){$-2$};
\node(a)at(7,0.35){$-2$}; 
\node(a)at(8,0.35){$-b$};
\node(a)at(9,0.35){$-5$};

\node(a)at(1,-1.5){$q_1=q_2=1$};
\node(a)at(8,-1.5){$q_1=2$ and $q_2=1$ };

\draw[thin](0,-3)--(1,-3);
\draw[thin](1,-3)--(2,-3);
\draw[thin](2,-3)--(3,-3);
\draw[thin](1,-3)--(1,-4);

\draw[fill=white](0,-3)circle(2pt);
\draw[fill=white](1,-3)circle(2pt);
\draw[fill=white](2,-3)circle(2pt);
\draw[fill=white](1,-4)circle(2pt);
\draw[fill=white](3,-3)circle(2pt);

\node(a)at(0,-2.65){$-3$};
\node(a)at(0.60,-4){$-2$};
\node(a)at(1,-2.65){$-b$};
\node(a)at(2,-2.65){$-3$}; 
\node(a)at(3,-2.65){$-2$};

\draw[thin](6,-3)--(7,-3);
\draw[thin](7,-3)--(8,-3);
\draw[thin](8,-3)--(8,-4);
\draw[thin](8,-3)--(9,-3);
\draw[thin](9,-3)--(10,-3);

\draw[fill=white](6,-3)circle(2pt);
\draw[fill=white](7,-3)circle(2pt);
\draw[fill=white](8,-4)circle(2pt);
\draw[fill=white](8,-3)circle(2pt);
\draw[fill=white](9,-3)circle(2pt);
\draw[fill=white](10,-3)circle(2pt);

\node(a)at(7.60,-4){$-2$};
\node(a)at(7,-2.65){$-2$};
\node(a)at(6,-2.65){$-2$}; 
\node(a)at(8,-2.65){$-b$}; 
\node(a)at(9,-2.65){$-3$};
\node(a)at(10,-2.65){$-2$};

\node(a)at(1.5,-4.5){$q_1= 1$ and $q_2=2$};
\node(a)at(8,-4.5){$q_1= q_2=2$ };

\end{tikzpicture}
\end{center}

\begin{center}
\begin{tikzpicture}

\draw[thin](0,-6)--(1,-6);
\draw[thin](1,-6)--(2,-6);
\draw[thin](2,-6)--(3,-6);
\draw[thin](1,-6)--(1,-7);

\draw[fill=white](0,-6)circle(2pt);
\draw[fill=white](1,-6)circle(2pt);
\draw[fill=white](2,-6)circle(2pt);
\draw[fill=white](1,-7)circle(2pt);
\draw[fill=white](3,-6)circle(2pt);

\node(a)at(0,-5.65){$-3$};
\node(a)at(0.60,-7){$-2$};
\node(a)at(1,-5.65){$-b$};
\node(a)at(2,-5.65){$-2$}; 
\node(a)at(3,-5.65){$-3$};

\draw[thin](6,-6)--(7,-6);
\draw[thin](7,-6)--(8,-6);
\draw[thin](8,-6)--(8,-7);
\draw[thin](8,-6)--(9,-6);
\draw[thin](9,-6)--(10,-6);

\draw[fill=white](6,-6)circle(2pt);
\draw[fill=white](7,-6)circle(2pt);
\draw[fill=white](8,-7)circle(2pt);
\draw[fill=white](8,-6)circle(2pt);
\draw[fill=white](9,-6)circle(2pt);
\draw[fill=white](10,-6)circle(2pt);

\node(a)at(7.60,-7){$-2$};
\node(a)at(7,-5.65){$-2$};
\node(a)at(6,-5.65){$-2$}; 
\node(a)at(8,-5.65){$-b$}; 
\node(a)at(9,-5.65){$-2$};
\node(a)at(10,-5.65){$-3$};

\node(a)at(1.5,-7.5){$q_1= 1$ and $q_2=3$};
\node(a)at(8,-7.5){$q_1=2$ and $q_2=3$ };

\draw[thin](0,-9)--(1,-9);
\draw[thin](1,-9)--(2,-9);
\draw[thin](2,-9)--(3,-9);
\draw[thin](1,-9)--(1,-10);
\draw[thin](3,-9)--(4,-9);
\draw[thin](4,-9)--(5,-9);

\draw[fill=white](0,-9)circle(2pt);
\draw[fill=white](1,-9)circle(2pt);
\draw[fill=white](2,-9)circle(2pt);
\draw[fill=white](1,-10)circle(2pt);
\draw[fill=white](3,-9)circle(2pt);
\draw[fill=white](4,-9)circle(2pt);
\draw[fill=white](5,-9)circle(2pt);

\node(a)at(0,-8.65){$-3$};
\node(a)at(0.60,-10){$-2$};
\node(a)at(1,-8.65){$-b$};
\node(a)at(2,-8.65){$-2$}; 
\node(a)at(3,-8.65){$-2$}; 
\node(a)at(4,-8.65){$-2$}; 
\node(a)at(5,-8.65){$-2$};

\draw[thin](6,-9)--(7,-9);
\draw[thin](7,-9)--(8,-9);
\draw[thin](8,-9)--(8,-10);
\draw[thin](8,-9)--(9,-9);
\draw[thin](9,-9)--(10,-9);
\draw[thin](10,-9)--(11,-9);
\draw[thin](11,-9)--(12,-9);

\draw[fill=white](6,-9)circle(2pt);
\draw[fill=white](7,-9)circle(2pt);
\draw[fill=white](8,-10)circle(2pt);
\draw[fill=white](8,-9)circle(2pt);
\draw[fill=white](9,-9)circle(2pt);
\draw[fill=white](10,-9)circle(2pt);
\draw[fill=white](11,-9)circle(2pt);
\draw[fill=white](12,-9)circle(2pt);

\node(a)at(7.60,-10){$-2$};
\node(a)at(7,-8.65){$-2$};
\node(a)at(6,-8.65){$-2$}; 
\node(a)at(8,-8.65){$-b$}; 
\node(a)at(9,-8.65){$-2$};
\node(a)at(10,-8.65){$-2$};
\node(a)at(11,-8.65){$-2$};
\node(a)at(12,-8.65){$-2$};

\node(a)at(1.5,-10.5){$q_1= 1$ and $q_2=4$};
\node(a)at(8,-10.5){$q_1=2$ and $q_2=4$ };

\end{tikzpicture}
\end{center}
\begin{center}
\begin{tabular}{lc}\toprule
 & $\mul(\im F)$ \\ \midrule
$q_1= q_2=1$ & $4+b $ \\
$q_1= 2$ and $q_2= 1$ & $3+b$ \\
$q_1= 1$ and $q_2=2$ & $2+b$ \\
$q_1= q_2=2$ & $1+b$ \\
$q_1= 1$ and $q_2=3$ & $2+b$ \\
$q_1= 2$ and $q_2=3$ & $1+b$ \\
$q_1= 1$ and $q_2=4$ & $1+b$ \\ 
$q_1= 2$ and $q_2=4$ & $b$ \\ \bottomrule 
\end{tabular}
\end{center}
In this case it is also only the simple singularity $E_8$ ($b=2$ in the last graph) that is a hypersurface.

The fundamental group $\pi_1(\im F) = \Z/{m\Z} \times I^*$ where
$I^*$ is the binary icosahedral group and
$m= 30b-15-10q_1-6q_2$ according to Theorem 2 section 6.2 of
\cite{orlik}. Again be aware that he uses the other convention for the
Seifert invariants, hence the different formula for $m$. 

\subsection{Maps in the case of the Binary Polyhedral groups}

In this section we find concrete maps to the quotients of $(\C^2,0)$ by
the binary polyhedral groups acting. To do so we calculate quotients
by finding invariant polynomials as described in the beginning of this
section. The invariant polynomials of the binary polyhedral groups were
originally found by Felix Klein in 1884. Here with the help of
Guilhermo Pe\~nfort-Sanchis and Mathematica we have created a set of
of generators of the ring of invariant polynomials.

First some words about how to obtain them. Remember that the binary
groups are the preimages of the polyhedral groups under a double cover
$S^3\to SO(3)$. This is not a topological cover, but a lie group
cover. Now the lie groups structure on $S^3$ is usually considered as
the group spin $1$, but one can also consider it as the groups of unit
quaternions. The unit quaternions are isomorphic to $SU(2)$ where the
generators or the quaternions are represented by
$\begin{psmallmatrix} 1 & 0 \\ 0 & 1 \end{psmallmatrix}$,
$\begin{psmallmatrix} i & 0 \\ 0 & -i \end{psmallmatrix}$,
$\begin{psmallmatrix} 0 & 1 \\ -1 & 0 \end{psmallmatrix}$  
$\begin{psmallmatrix} 0 & i \\ i & 0 \end{psmallmatrix}$  
So the action of the binary polyhedral groups we will use is in this
embedding in $SU(2)$.

The binary dihedral group $D_{4n}^*$ can be generated by 
$\begin{psmallmatrix} e^{\pi i/n} & 0 \\ 0 &
  e^{-\pi i/n} \end{psmallmatrix}$ and $\begin{psmallmatrix} 0 & 1 \\ -1 &
  0 \end{psmallmatrix}$. The first generator transforms $u$ to
$e^{\pi i/n}u$ and $v$ to $e^{-\pi i/n}v$ and the second transforms $u$
to $-v$ and $v$ to $u$. This means that the monomial $uv$ is invariant
under the action of the first generator, but transformed into $-uv$ by
the second. Hence $p_{D1}= (uv)^2$ is invariant under both generators and thus
under the whole group. Now $u^{2n}$ and $v^{2n}$ are invariant under
the first generator but interchanged by the second. So $p_{D2} = u^{2n}+v^{2n}$
is invariant under the whole group. This also means that
$u^{2n}-v^{2n}$ is invariant under the first generator but changes
sign under the second, so $p_{D3} = uv(u^{2n}-v^{2n})$ is invariant. 

We have
now found $3$ different invariant polynomials, and it is not hard to
see that they are algebraically independent. So we have a generating
set of the ring of invariants. If we let $x=(uv)^2$, $y=u^{2n}+v^{2n}$
and $z=uv(u^{2n}-v^{2n})$ then we get the relation $z^2 =
(uv)^2(u^{4n}+v^{4n}-2(uv)^{2n}) =x(y^2-4x^n)$. Since we know that
the quotient is a hypersurface singularity and this is an irreducible
polynomial, this is an equation for the image given by the map
$(u,v)\mapsto \bigl((uv)^2,u^{2n}+v^{2n},uv(u^{2n}-v^{2n})\bigr)$.

The cases of $T^*$, $O^*$ and $I^*$ is similar. The generators as a
subgroup of $SU(2)$ are 
\begin{align*}
&T^*&
&\tfrac{1}{2}\begin{pmatrix} 1+i &
  1+i \\ -1+i & 1+i \end{pmatrix},\ \tfrac{1}{2}\begin{pmatrix} 1+i &
  1-i \\ -1-i & 1-i \end{pmatrix} \\
&O^*&
&\tfrac{1}{2}\begin{pmatrix} 1+i &
  1+i \\ -1+i & 1+i \end{pmatrix},\
\tfrac{1}{\sqrt{2}}\begin{pmatrix} 0& 1+i \\ -1+i &
  0 \end{pmatrix} \\ 
&I^*&
&\tfrac{1}{2}\begin{pmatrix} 1+i &
  1+i \\ -1+i & 1+i \end{pmatrix},\ \tfrac{1}{2}\begin{pmatrix} 
  \tfrac{2}{1+\sqrt{5}} +i\tfrac{1+\sqrt{5}}{2} & 1\\ -1 &
  \tfrac{2}{1+\sqrt{5}} +i\tfrac{1+\sqrt{5}}{2} \end{pmatrix}
\end{align*}It is not difficult to see that these matrices
satisfies the relations of the generators of their respective groups. To
find the invariant polynomials we do the same process as in the binary
dihedral case. Testing polynomials until we have three algebraically
independent invariant polynomials. In this case the calculations
becomes increasingly harder and we used the computer algebra system
Matematica to aid us. In the later case even if we made assumption
about the polynomials being symmetric or antisymmetric Matematica had
problems solving them, and we had to simplify the expression by using 
partial solutions to get the program to give a solution.

This gave us the following list of invariant polynomials:
\begin{align*}
&T^*&
p_{T1}=\,&uv^5-u^5v\\ & & p_{T2}=\,&u^8+v^8+14u^4v^4\\ & &
p_{T3}=\,&u^{12}+v^{12}-33u^8v^4-33u^4v^8 \\
&O^*&
p_{O1}=\,&u^{10}v^2+u^2v^{10}-2u^6v^6\\ & & p_{O2}=\,&u^8 +v^8 +14u^4v^4\\ 
& & P_{O3}=\,&34u^5v^{13}-34u^{13}v^5+u^{17}v-uv^{17} \\ 
&I^*&
p_{I1}=\,&\sqrt{5}u^{12}+\sqrt{5}v^{12}-22u^{10}v^2 -22u^2v^{10}
-33\sqrt{5}u^8v^4 -33\sqrt{5}u^4v^8 +44u^6v^6\\ & & p_{I2}=\,&-3(u^{20}+v^{20})
  -38\sqrt{5}(u^{18}v^2+u^2v^{18}) +57(u^{16}v^4+u^4v^{16})\\
  & & &-456\sqrt{5}(u^{14}v^6+u^6v^{14}) +1482(u^{12}v^2+u^2v^{12})
  +988\sqrt{5}u^{10}v^{10}\\ & & p_{I3}=\,&225(u^{29}v-uv^{29})
  +580\sqrt{5}(u^{27}v^3-u^3v^{27}) +15921(u^{25}v^5-u^5v^{25})\\
  & & &-20880\sqrt{5}(u^{23}v^7-u^7v^{29}) +90045(u^{21}v^9-u^9v^{25})\\
  & & &+40020\sqrt{5}(u^{19}v^{11}-u^{11}v^{19})
  +570285(u^{17}v^{13}-u^{13}v^{17}).
\end{align*}

Finding the equations by hand would be rather tedious, so instead we
used the following observation. Let $\psi:\C[x,y,z]\to \C[u,v]$ be the
map defined by sending $x\mapsto p_{*1}$, $y\mapsto p_{*2}$ and $z\mapsto
p_{*3}$, where $*$ is $D$, $T$, $O$ or $I$ respectively. Then the ring
of invariants of the action is isomorphic to
$\C[x,y,z]/\ker\psi$. We then used Singular and OSCAR to calculate a 
generator of $\ker\psi$ in each of the cases. The result is summarized
in the table below.

\begin{center}
\begin{tabular}{lcc}\toprule
Singularity & Map &  Equation of the image \\ \midrule
$D_n$ & 
$(p_{D1},p_{D2},p_{D3})$ & $x(y^2-4x^n)-z^2=0 $\\
$E_6$ & 
 $(p_{T1},p_{T2},p_{T3})$ & $108x^4 -y^3 +z^2= 0$ \\
$E_7$ & 
$(p_{O1},p_{O2},p_{O3})$ & $108x^3-xy^3+z^2 = 0$ \\
$E_8$ & 
$(p_{I1},p_{I2},p_{I3})$ & $27x^5+25\sqrt{5}y^3+4z^2=0$ \\ \bottomrule 
\end{tabular}
\end{center}


\subsection{Maps for the product of Cyclic and Binary polyhedral groups}\label{mapstobinarypolyhedral}

Consider the case of the fundamental group being $\Z/m\Z\times G$
where $G$ is a binary polyhedral group and $m>1$. In this case we cannot find
the groups acting as subgroups of $SU(2)$, since the only finite
subgroups of $SU(2)$ are the binary polyhedral groups. Instead we
will find them acting as subgroups of of $U(2)$. If $A,B$ are
subgroups of some group $C$, then the set $AB=\{ ab\in C\, \vert\,
a\in A\text{ and }b\in B\}$ is a subgroups if either $A$ or $B$ is
normal in $C$. If furthermore $A\cap B = \{ \id \}$ and both $A$ and
$B$ are normal in $AB$, then $AB$ is isomorphic to $A\times B$.

Let $G\subset SU(2)\subset U(2)$ be one of the binary polyhedral groups
considered above, and consider $\Z/m\Z\subset U(2)$ as the subgroup
generated by $\begin{psmallmatrix} e^{2\pi i/m} & 0 \\ 0 & e^{2\pi
    i/m} \end{psmallmatrix}$. Then $\Z/m\Z$ is in the center of $U(2)$
and hence normal so $(\Z/m\Z)G$ is a subgroup of $U(2)$. $\Z/m\Z$ is
still in the center of $(\Z/m\Z)G$ so it follows that $\Z/m\Z$ is
normal in $(\Z/m\Z)G$. Since every element in $(\Z/m\Z)G$ consist of
products of elements that are in $G$ or commutes with $G$ it follows
that $G$ is also normal in $(\Z/m\Z)G$. Lastly, since elements in
$\Z/m\Z$ has order a divisor of $m$, and by construction $m$ is
coprime with the order of $G$ we have that  $(\Z/m\Z)\cap G = \{ \id
\}$. So we get that the group generated by $\begin{psmallmatrix}
  e^{2\pi i/m} & 0 \\ 0 & e^{2\pi i/m} \end{psmallmatrix}$ and the
generators of $G$ is the fundamental group of the link.

To find the invariant polynomials we just have to intersect the
invariant polynomials we found for the given binary polyhedral group
with the invariant polynomials for the given action of $\Z/m\Z$. In
\ref{mapsforcyclicgroups} we found this to be generated by all
monomials on the form $u^av^b$ where $a+b=m$. Notice that in this case
the action is simply given by multiplying each variable by a $m$ root
of unity $\xi_m$. This implies that the action will transform any
monomial of degree $d$ by multiplying with $\xi_m^d$. It will then
of course also transform a homogeneous polynomial of degree $d$ by
multiplying it with $\xi_m^d$. So any homogeneous polynomial is
invariant under the action of $\Z/m\Z$ if and only if it have degree a
multiple of $m$. Furthermore any invariant polynomial will have that
any of its homogeneous components are invariant under the action, and
in particular any monomial in an invariant polynomial have degree a
multiple of $m$.

Let $p_{*1}$, $p_{*2}$ and $p_{*3}$ be the generators of the
invariant polynomials under the action of the give binary polyhedral
group. Let $d_{*i}$ be the degree of $p_{*i}$, then  because $p_{*i}$ is
homogenous the action of the generator of $\Z/m\Z$ will just be
multiplication by $\xi_m^{d_{*i}}$. An invariant polynomial of the form
$p_{*1}^{a_1}p_{*2}^{a_2}p_{*3}^{a_3}$ will be transformed into
$\xi_m^{a_1d_{*1}+a_2d_{*2}+a_3d_{*3}}p_{*1}^{a_1}p_{*2}^{a_2}p_{*3}^{a_3}$, and hence be
invariant if and only if $a_1d_{*1}+a_2d_{*2}+a_3d_{*3}$ is a multiple of
$m$. When we want to find generator we can furthermore assume that
$m^2$ does not divide $a_1d_{*1}+a_2d_{*2}+a_3d_{*3}$. Since any
invariant polynomial under the action of the binary polyhedral group must be
a polynomial in $d_{*1}$, $d_{*2}$ and $d_{*3}$ we get that the
invariant polynomials under the product group is generated by monomial
of the form $p_{*1}^{a_1}p_{*2}^{a_2}p_{*3}^{a_3}$ with
$a_1d_{*1}+a_2d_{*2}+a_3d_{*3}$ a multiple of $m$.

\subsubsection{$G$ is the Binary Dihedral group}
In the case of $D_{4n}^*$ the degrees of the generators are $d_{D1}=4$,
$d_ {D2}=2n$ and $d_{D3}=2n+1$ and we have that $m=(b-1)n-q$ with
$\gcd(m,2n)=1$ and $\gcd(n,q)=1$. So the
equation defining the power of the invariant polynomials is
$4a_1+2na_2+(2n+2)a_3 = lm =l\bigl((b-1)n-q\bigr)$ for some $l$. Since
$m$ is odd we can assume that $l$ is even and replace with $l'=l/2$
and get $2a_1+na_2+(n+1)a_3 = l'm =l'\bigl((b-1)n-q\bigr)$. For given
values of $n$, $b$ and $q$ this equation is not hard to solve.

\begin{example}
Assume that $b=3$, $n=2$ and $q=1$ the equation becomes
$2a_1+2a_2+3a_3=3l'$. The solutions for
$(a_1,a_2,a_3)$ are $(3,0,0)$, $(2,1,0)$, $(1,2,0)$, $(0,3,0)$ and
$(0,0,1)$, and we get the polynomials $(vv)^6$, $(uv)^4(u^4+v^4)$,
$(uv)^2(u^4+v^4)^2$, $(u^4+v^4)^3$ and $uv(u^4-v^4)$. But notice that
$(uv(u^4-v^4))^2+4(uv)^6   = (uv)^2(u^8+v^8-2(uv)^4) +4(uv)^6 =
(uv)^2(u^8+v^8+2(uv)^4)=(uv)^2(u^4+v^4)^2$. So we only have the four
algebraically independent generators $(uv)^6$, $(uv)^4(u^4+v^4)$,
 $(u^4+v^4)^3$ and $uv(u^4-v^4)$. The equations can again be found as
 the kernel of a ring homomorphism which we here have calculated using
 singular and get the equations of the image $yw^2+4xy-xz=0$,
 $xw^2+4x^2-y^2+0$ and $w^4-16x^2+8y^2-yz=0$, for the map $F(u,v) =
 \bigl((uv)^6,(uv)^4(u^4+v^4),(u^4+v^4)^3,uv(u^4-v^4)\bigr)$. 
\end{example}

\subsubsection{$G$ is the Binary Tetrahedral group}
In the case of $T^*$ the degrees of the generating invariant
polynomials are $d_{T1}=6$, $d_{T2}=8$ and $d_{T3}=12$, $m=6b-3-2q_1-2q_2$ and
$\gcd(m,12)=1$. So the equation defining products of $p_{T1}$, $p_{T2}$ and
$p_{T3}$ being invariant under the action of $\Z/m\Z$ is
$6a_1+8a_2+12a_3 = lm =l\bigl(6b-3-2q_1-2q_2\bigr)$. Since $m$ is
always odd we can replace $l$ by $l'=l/2$ and reduce the equation to
$3a_1+4a_2+6a_3 = l'm =l' \bigl(6b-3-2q_1-2q_2\bigr)$. Again for
concrete values of $b$, $q_1$ and $q_2$ solving this is easy.

\begin{example}
Assume that $b=3$, $q_1=1$ and $q_2=1$ the equation becomes
$3a_1+4a_2+6a_3=5l'$. The solutions for
$(a_1,a_2,a_3)$ are $(5,0,0)$, $(3,0,1)$, $(2,1,0)$, $(1,3,0)$,
$(1,0,2)$, $(0,5,0)$, $(0,1,1)$ and $(0,0,5)$. Again this set is not
minimal, since we know that in this case the embedding dimension is
$6$. We have that $p_{p2}^5 = p_{T2}p_{T3} + 108 p_{T1}^2p_{T2}$ and $p_{T3}^5 =
p_{T2}p_{T3}-108P_{T1}^3p_{T3}^3P_{T1}p_{T2}^3 +11664
p_{T1}^5p_{T1}^3p_{T2}$.  So a map with image the singularity with the
given topology is given by 
\begin{align*}
(x_1,x_2,x_3,x_4,x_5,x_6) = (p_{T1}^5, p_{T1}^3p_{T3}, p_{T1}^2p_{T1},
p_{T1}p_{T2}^3, 
p_{T1}p_{T3}^2, p_{T2}p_{T3}). 
\end{align*}Using the computer algebra system
OSCAR we find by the method used earlier that a set of equations for
the image is
\begin{align*}
&x_1x_6 - x_2x_3 = 0,& &108x_1x_3 + x_2x_6 - x_3x_4 = 0,& &108x_1x_2x_4 + x_1x_6^3 - x_2x_4^2 = 0,&\\ &x_2x_4 -
x_3^2x_6 = 0,& &108x_1x_4 + x_3x_6^2 - x_4^2 = 0,& &108x_1x_4^2 + 108x_2^2x_4 + 
x_2x_6^3 - x_4^3 = 0,&\\ &x_1x_4 - x_3^3 =
0,&  &108x_1^2 - x_1x_4 + x_2^2 = 0,&  
&3456x_1x_2x_4 + 3456x_2^3 - 32x_2x_4^2 + x_5^5 = 0.&
\end{align*}  
\end{example}

\subsubsection{$G$ is the Binary Octahedral group}
For the actions involving $O^*$ the degrees of the generating
polynomials are $d_{O1} = 12$, $d_{O2} = 8$, $d_{O3} = 18$ and $m=
12b-6-4q_1-3q_2$. Notice that $\gcd(m,6)=1$ so $m$ is always odd and we
can divide the equation by 2. This makes the equation
defining the powers of the invariant monomials $6a_1+4a_2+9a_3=lm =
l(12b-6-4q_1-3q_2)$. Again for given value of $b$, $q_1$ and $q_2$
solving this is easy.

\begin{example}
Let $b=2$, $q_1=2$ and $q_2=1$. Then the equation defining the powers
is  $6a_1+4a_2+9a_3=7l$. Solutions for $(a_1,a_2,a_3)$ are
$(7,0,0)$, $(4,1,0)$, $(2,0,1)$, $(1,2,0)$, $(1,0,4)$, $(0,7,0)$, 
$(0,3,1)$, $(0,2,3)$, $(0,1,5)$, $(0,0,7)$. Now this is not a minimal
set of generating invariant polynomials. Using computer algebra one
can easily see that 
\begin{align*}
11664p_{O1}^7 =& p_{O1}^3p_{O2}^6 -
108p_{O1}^4p_{O3}^2 - p_{O1}^2p_{O2}^3p_{O3}^2,\\ p_{O1}p_{O3}^4 =&
-108p_{O1}^4p_{O3}^2 + p_{O1}^2p_{O2}^3p_{O3}^2,\\ p_{O2}^2p_{O3}^3 =&
p_{O1}p_{O2}^5p_{O3}^5 - 108 p_{O1}^3p_{O2}^3p_{O3},\\ p_{O2}p_{O3}^5
=&  p_{O1}^2p_{O2}^7p_{O3} -216 p_{O1}^4p_{O2}^4p_{O3} +11664
p_{O1}^6p_{O2}p_{O3},\\ p_{O3}^7 =& p_{O1}^2p_{O2}^3p_{O3}^2 - 216
p_{O1}^4p_{O2}(p_{O1}p_{O2}^5p_{O3}^5 - 108 p_{O1}^3p_{O2}^3p_{O3})
+11664 p_{O1}^6P_{O3}^3.
\end{align*}This eliminates $5$ of the $10$ monomials in
$P_{O1}$, $P_{O2}$ and $P_{O3}$ and since the embedding dimension is
$5$ the remaining $5$ give the generators of the invariant
polynomials. Hence a map is given by:
\begin{align*}
(x_1,x_2,x_3,x_4,x_5) = (p_{O1}^4p_{O2}, p_{O1}^2p_{O3},
p_{O1}p_{O2}^2, p_{O2}^7, p_{O2}^3p_{O3}).
\end{align*}Using OSCAR we calculate the equations of the image to be:
\begin{align*}
&x_2x_4 - x_3^2x_5=0,\  x_1x_5^2 - x_2^2x_4=0,\ x_1x_5 - x_2x_3^2=0,\\
&11664x_1x_3 + 108x_2x_5 - x_3x_4 + x_5^2=0,\\
 &11664x_1^2 - x_1x_4 + 108x_2^2x_3 + x_2x_3x_5=0,\
108x_1x_4x_5 - x_2x_4^2 + x_3x_5^3=0,\\ &11664x_1x_2x_4 + 108x_2x_3x_5^2 -
x_2x_4^2 + x_3x_5^3=0,\ 108x_1x_4 - x_3^2x_4 + x_3x_5^2=0,\\ &11664x_1x_2x_4^2 -
216x_1x_4^2x_5 + x_2x_4^3 - x_5^5=0,\ 108x_2x_3x_4 - x_3x_4x_5 +
x_5^3=0,\\ &11664x_1^2x_4x_5 -
216x_1x_2x_4^2 + x_1x_4^2x_5 - x_2x_5^4=0,\  108x_1x_3 + x_2x_5 -
x_3^3=0, \\ &629856x_1^2x_4 - 54x_1x_4^2 -
54x_2x_5^3 + x_3x_4x_5^2 - x_5^4=0,\\ &136048896x_1^3 -
23328x_1^2x_4 + x_1x_4^2 - 
11664x_2^3x_5 - 216x_2^2x_5^2 - x_2x_5^3=0,\\ &136048896x_1^2x_2x_4 +
23328x_1x_2x_4^2 - 11664x_2^2x_5^3 + x_2x_4^3 - 216x_2x_5^4 - x_5^5=0,\\
&136048896x_1^3x_4 - 23328x_1^2x_4^2 - 11664x_1x_2x_5^3 + x_1x_4^3 +
216x_1x_5^4 - x_2x_4x_5^3 = 0.
\end{align*}
\end{example}

\subsubsection{$G$ is the Binary Icosahedral group}
Lastly for the case of $I^*$ the degrees of the generating polynomials
are $d_{I1}= 12$, $d_{I2}= 20$, $d_{I3}=30$ and $m=30b-15-10q_1-6q_2$. Notice
that $\gcd(m,30)=1$ so $m$ is odd and we can divide by $2$ in the
equation for finding the powers of the invariant monomials. This gives
us $6a_1+20a_2+15a_3=lm=l(30b-15-10q_1-6q_2)$. As in the other cases
finding the solution for given values of  $b$, $q_1$ and $q_2$ should
not be any problem.


\subsection{The Non Binary polyhedral case}

In the case of the group being $D'_{2^{k+2}(2l+1}$ or $T*_{8\cdot
  3^k}$ we cannot find embedding of the group into $SU(2)$
since the only finite subgroups of $SU(2)$ are the cyclic groups and
the binary polyhedral groups. So the representation of the groups have
to be in $U(2)$. These groups are not very studied in the literature,
so we could not find concrete representation already given to
work with. So no example will be given, but the process is the same
find an embedding of the group into $U(2)$ and calculate the invariant
of this action on $\C[u,v]$. Then represent $\Z/m\Z$ by the diagonal
action by $\begin{psmallmatrix}
  e^{2\pi i/m} & 0 \\ 0 & e^{2\pi i/m} \end{psmallmatrix}$ which will
because of the conditions on $m$ commute with the action of
$D'_{2^{k+2}(2l+1}$ or $T*_{8\cdot 3^k}$ so the invariant of the
product action is just the intersection of the invariants of the two
actions.

\bibliography{general}

\def\cprime{$'$}
\providecommand{\bysame}{\leavevmode\hbox to3em{\hrulefill}\thinspace}
\providecommand{\MR}{\relax\ifhmode\unskip\space\fi MR }
\providecommand{\MRhref}[2]{%
  \href{http://www.ams.org/mathscinet-getitem?mr=#1}{#2}
}
\providecommand{\href}[2]{#2}
\begin{thebibliography}{MNnB20}

\bibitem[Bri68]{BrieskornRationalSingularitiesOfComplexSurfaces}
Egbert Brieskorn, \emph{Rationale singularit\"{a}ten komplexer fl\"{a}chen.},
  Inventiones mathematicae \textbf{4} (1967/68), 336--358.

\bibitem[DFN85]{DubrovinFomenkoNovikovModernGeometryII}
B.~A. Dubrovin, A.~T. Fomenko, and S.~P. Novikov, \emph{Modern
  geometry---methods and applications. {P}art {II}}, Graduate Texts in
  Mathematics, vol. 104, Springer-Verlag, New York, 1985, The geometry and
  topology of manifolds, Translated from the Russian by Robert G. Burns.
  \MR{807945}

\bibitem[dJvS91]{dejongvanstraten}
T.~de~Jong and D.~van Straten, \emph{Disentanglements}, Singularity theory and
  its applications, {P}art {I} ({C}oventry, 1988/1989), Lecture Notes in Math.,
  vol. 1462, Springer, Berlin, 1991, pp.~199--211. \MR{1129033}

\bibitem[DV44]{duval}
Patrick Du~Val, \emph{On absolute and non absolute singularities of algebraic
  surfaces}, Rev. Fac. Sci. Univ. Istanbul (A) \textbf{11} (1944), 159--215
  (1946). \MR{31286}

\bibitem[Fuk82]{fukuda}
Takuo Fukuda, \emph{Local topological properties of differentiable mappings.
  {I}}, Invent. Math. \textbf{65} (1981/82), no.~2, 227--250. \MR{641129}

\bibitem[Hem76]{hempel}
John Hempel, \emph{{$3$}-{M}anifolds}, Annals of Mathematics Studies, No. 86,
  Princeton University Press, Princeton, NJ; University of Tokyo Press, Tokyo,
  1976. \MR{415619}

\bibitem[Hop26]{hopfOnTheClifford-KleinSpaceProblem}
Heinz Hopf, \emph{Zum {C}lifford-{K}leinschen {R}aumproblem}, Math. Ann.
  \textbf{95} (1926), no.~1, 313--339. \MR{1512281}

\bibitem[JN83]{brandies}
Mark Jankins and Walter~D. Neumann, \emph{Lectures on {S}eifert manifolds},
  Brandeis Lecture Notes, vol.~2, Brandeis University, Waltham, MA, 1983.
  \MR{MR741334 (85j:57015)}

\bibitem[Joh79]{johannson}
Klaus Johannson, \emph{Homotopy equivalences of {$3$}-manifolds with
  boundaries}, Lecture Notes in Mathematics, vol. 761, Springer, Berlin, 1979.
  \MR{551744}

\bibitem[JS79]{jacosalen}
William~H. Jaco and Peter~B. Shalen, \emph{Seifert fibered spaces in
  {$3$}-manifolds}, Mem. Amer. Math. Soc. \textbf{21} (1979), no.~220,
  viii+192. \MR{539411}

\bibitem[Lau72]{lauferrational}
Henry~B. Laufer, \emph{On rational singularities}, Amer. J. Math. \textbf{94}
  (1972), 597--608. \MR{0330500 (48 \#8837)}

\bibitem[Lau73]{laufertauttwodsing}
\bysame, \emph{Taut two-dimensional singularities}, Math. Ann. \textbf{205}
  (1973), 131--164. \MR{333238}

\bibitem[MNnB20]{mondnunoballesterossingularitiesofmappings}
David Mond and Juan~J. Nu\~{n}o Ballesteros, \emph{Singularities of
  mappings---the local behaviour of smooth and complex analytic mappings},
  Grundlehren der mathematischen Wissenschaften [Fundamental Principles of
  Mathematical Sciences], vol. 357, Springer, Cham, [2020] \copyright 2020.
  \MR{4321457}

\bibitem[Mon85]{mond1}
David Mond, \emph{On the classification of germs of maps from
  {${\mathbbm{R}}^2$} to {${\mathbbm{R}}^3$}}, Proc. London Math. Soc. (3)
  \textbf{50} (1985), no.~2, 333--369. \MR{772717}

\bibitem[Mon95]{mond2}
\bysame, \emph{Looking at bent wires---{${\mathscr{A}}_e$}-codimension and the
  vanishing topology of parametrized curve singularities}, Math. Proc.
  Cambridge Philos. Soc. \textbf{117} (1995), no.~2, 213--222. \MR{1307076}

\bibitem[Mum65]{mumfordgit}
D.~Mumford, \emph{Geometric invariant theory}, Ergeb. Math. Grenzgeb., vol.~34,
  Springer-Verlag, Berlin, 1965 (English).

\bibitem[N\'22]{nemethinormalsurfacesing}
Andr\'{a}s N\'{e}methi, \emph{Normal surface singularities}, Ergebnisse der
  Mathematik und ihrer Grenzgebiete. 3. Folge. A Series of Modern Surveys in
  Mathematics [Results in Mathematics and Related Areas. 3rd Series. A Series
  of Modern Surveys in Mathematics], vol.~74, Springer, Cham, [2022] \copyright
  2022. \MR{4510934}

\bibitem[Neu81]{plumbing}
Walter~D. Neumann, \emph{A calculus for plumbing applied to the topology of
  complex surface singularities and degenerating complex curves}, Trans. Amer.
  Math. Soc. \textbf{268} (1981), no.~2, 299--344. \MR{MR632532 (84a:32015)}

\bibitem[Neu07]{Neu07}
\bysame, \emph{Graph 3-manifolds, splice diagrams, singularities}, Singularity
  theory, World Sci. Publ., Hackensack, NJ, 2007, pp.~787--817. \MR{MR2342940
  (2008k:32085)}

\bibitem[Orl72]{orlik}
Peter Orlik, \emph{Seifert manifolds}, Lecture Notes in Mathematics, Vol. 291,
  Springer-Verlag, Berlin, 1972. \MR{MR0426001 (54 \#13950)}

\bibitem[Ped11]{myarticle}
Helge~M{\o}ller Pedersen, \emph{Splice diagram determining singularity links
  and universal abelian covers}, Geom. Dedicata \textbf{150} (2011), 75--104.
  \MR{2753699}

\bibitem[Pri67]{PrillLocalClasificationsOfQuotients}
David Prill, \emph{Local classification of quotients of complex manifolds by
  discontinuous groups}, Duke Math. J. \textbf{34} (1967), 375--386.
  \MR{210944}

\bibitem[Sco83]{scott}
Peter Scott, \emph{The geometries of {$3$}-manifolds}, Bull. London Math. Soc.
  \textbf{15} (1983), no.~5, 401--487. \MR{MR705527 (84m:57009)}

\bibitem[Sei33]{seifertTopologyofThreeDimensionalFibreSpaces}
H.~Seifert, \emph{Topologie {D}reidimensionaler {G}efaserter {R}\"aume}, Acta
  Math. \textbf{60} (1933), no.~1, 147--238. \MR{1555366}

\bibitem[Wal67]{waldhausen}
Friedhelm Waldhausen, \emph{Eine {K}lasse von {$3$}-dimensionalen
  {M}annigfaltigkeiten. {I}, {II}}, Invent. Math. \textbf{3} (1967), 308--333;
  ibid. 4 (1967), 87--117. \MR{235576}

\end{thebibliography}

\end{document}